\documentclass[11pt]{article}

\usepackage{amsmath}
\usepackage{amssymb}
\usepackage{amsthm}
\usepackage{amsfonts}
\usepackage{graphicx}
\usepackage{textcomp}
\usepackage{color}
\usepackage{hyperref}
\usepackage{multicol}
\usepackage{slashed}
\usepackage{setspace}
\usepackage[utf8]{inputenc}
\usepackage[english]{babel}
\usepackage[a4paper, total={5.7in, 9.5in}]{geometry}
\numberwithin{equation}{section}

\newtheorem{Proposition}{Proposition}[section]

\newtheorem{Corollary}{Corollary}[section]
\newtheorem{Definition}{Definition}[section]

\begin{document}

\begin{titlepage}

\title{`Anti-Commutable' Pre-Leibniz Algebroids and Admissible Connections}

\author{Tekin Dereli, Keremcan Do\u{g}an$^*$ \\ \small Department of Physics, Ko\c{c} University, 34450 Sar{\i}yer, \.{I}stanbul, Turkey}

\date{$^*$ Corresponding author, \mbox{E-mail: kedogan[at]ku.edu.tr}}

\maketitle

\begin{abstract}
\noindent The concept of algebroid is convenient as a basis for constructions of geometrical frameworks. For example, metric-affine and generalized geometries can be written on Lie and Courant algebroids, respectively. Furthermore, string theories might make use of many other algebroids such as metric algebroids, higher-Courant algebroids or conformal Courant algebroids. Working on the possibly most general algebroid structure, which generalizes many of the algebroids used in the literature, is fruitful as it creates a chance to study all of them at once. Local pre-Leibniz algebroids are such general ones in which metric-connection geometries are possible to construct. On the other hand, the existence of the `locality operator', which is present for the right-Leibniz rule for the bracket, necessitates the modification of torsion and curvature operators in order to achieve tensorial quantities. In this paper, this modification of torsion and curvature is explained from the point of view that the modification is applied to the bracket instead. This leads one to consider `anti-commutable' pre-Leibniz algebroids which satisfy an anti-commutativity-like property defined with respect to a choice of an equivalence class of connections. These `admissible' connections are claimed to be the necessary ones while working on a geometry of algebroids. This claim is due to the fact that one can prove many desirable properties and relations if one uses only admissible connections. For instance, for admissible connections, we prove the first and second Bianchi identities, Ricci identity, Cartan structure equations, Cartan magic formula, the construction of Levi-Civita connections, the decomposition of connection in terms of torsion and non-metricity. These all are possible because the modified bracket becomes anti-symmetric for an admissible connection so that one can apply the machinery of almost- or pre-Lie algebroids. We investigate various algebroid structures from the literature and show that they admit admissible connections which are metric-compatible in some generalized sense. Moreover, we prove that local pre-Leibniz algebroids that are not anti-commutable cannot be equipped with a torsion-free, and in particular Levi-Civita, connection.
\end{abstract}

\vskip 2cm


\thispagestyle{empty}

\end{titlepage}

\maketitle

\section{Introduction}
\noindent Algebroid structures have been increasingly studied in order to construct geometrical frameworks which can be utilized to create a suitable background for classical field theories in physics. For instance, metric-affine geometries on a smooth manifold can be written on the tangent bundle of a manifold where this tangent bundle equipped with the Lie bracket becomes a Lie algebroid. The usual geometrical notions can be easily generalized on an arbitrary Lie algebroid \cite{1}. Metric-affine geometries constitute the mathematical foundation for gravity theories such as general relativity or Einstein-Cartan gravity. On the other hand, a search for a bracket similar to the Lie bracket on the tangent bundle direct sum cotangent bundle leads to generalized geometry \cite{2}. The new bracket yields a Courant algebroid structure on the generalized tangent bundle \cite{3}. By using ideas very similar to generalized geometry, one can give the double field theory formulation of effective string theories corparating $T$-duality in a similar fashion as in the usual gravity theories \cite{4}, \cite{5}. The double field theory formulation turns out to be closely related to para-Hermitian structures in which Vaisman's metric algebroids are used \cite{6}. Furthermore, $U$-duality of string theories can be also geometrized in the framework of exceptional generalized geometry \cite{7} in which the generalization of higher-Courant algebroids provides the mathematical basis \cite{8}. There are also many other algebroids used in the recent literature such as dull algebroids \cite{9}, or conformal Courant algebroids \cite{10}, or omni-Lie algebroids \cite{11}.

All the algebroids above provide special cases of local pre-Leibniz algebroids which are defined with the least possible assumptions. It is then natural to ask for a more all-inclusive framework on these pre-Leibniz algebroids. In order to achieve this comprehensive scheme, metric-connection geometries on local pre-Leibniz algebroids are constructed in our previous paper \cite{12} by expanding the ideas of \cite{13}. Since pre-Leibniz algebroids are defined with respect to an arbitrary bracket which needs not necessarily be anti-symmetric, both the left- and right-Leibniz rules are given separately, where in the former case an extra term with a `locality operator' is included. This situation creates a difficulty for defining the torsion and curvature operators, so one needs to modify the usual forms of these maps in order to have tensorial quantities. The same difficulty shows up in all of the geometries above as they have non-vanishing locality operators in general. Because of this modification, the geometrical meaning of some quantities are not clear cut \cite{13}. One main motivation in this paper is to understand better why this modification is necessary and what it means from a more familiar geometrical point of view. We will show that this modification can be understood as the modification of the bracket instead of the torsion and curvature operators. This leads us to the definition of `anti-commutable' pre-Leibniz algebroids, where the bracket is modified in such a way that it becomes anti-symmetric, so that one can use the machinery on almost-Lie algebroids. This definition of `anti-commutability' depends on a choice of an equivalence class of $E$-connections, which we call `admissible'. Many algebroids in the literature turns out to be a special case of anti-commutable pre-Leibniz algebroids with the appropriate choice of admissible connections. As proven below, this admissibility condition is usually just a metric-compatibility condition in some generalized sense. Our most important claim in this paper is to only consider admissible linear $E$-connections while working on a geometry in an algebroid setting because the anti-symmetry of the modified bracket sustains many important properties of the usual metric-affine geometry. 

Organization of the paper is as follows: In the second section of the paper, we will briefly summarize the necessary geometrical information about the usual metric-affine geometry on a smooth manifold. This section will set the notation of the paper while defining and outlining important properties of torsion, curvature and non-metricity tensors. In the second section, these constructions will be carried out on local pre-Leibniz algebroids. Several different algebroid structures will be introduced, together with the anti-commutable pre-Leibniz algebroids. Some examples of such anti-commutable pre-Leibniz algebroids from the literature and admissible connections on them will be discussed. The modification of the torsion and curvature operators will be investigated in terms of the modification of the bracket. As mentioned already, the anti-symmetry of these new brackets leads to many crucial relations, including the first and second Bianchi identities, the Ricci identity, Cartan structure equations, decomposition of connections in terms of torsion and non-metricity, construction of Levi-Civita connections. 

It must be noted that every construction here will be assumed to be in the smooth category, and the Einstein's summation convention for repeated indices is used. Moreover, one should be aware that the definitions of some algebroids (especially the ones with almost and pre prefixes) in the literature differ from paper to paper, so one should be careful about the assumptions of the certain work.



\section{Metric-Affine Geometry on Smooth Manifolds}

\noindent Any (Hausdorff, paracompact, orientable) manifold $M$ comes with two naturally constructed fiber bundles; namely the tangent bundle $T(M)$ and its dual cotangent bundle $T^*(M)$. By using the tensor products of these bundles, one can get the set of $(q, r)$-type tensors, denoted by $Tens^{(q, r)}(M)$. The set of vector fields is denoted by $\mathfrak{X}(M)$, which is equipped with the Lie bracket $[\cdot,\cdot]: \mathfrak{X}(M) \times \mathfrak{X}(M) \to \mathfrak{X}(M)$, that makes $\mathfrak{X}(M)$ into a Lie algebra. The Lie bracket can be extended into whole tensor algebra by the Lie derivative $\mathcal{L}_V: Tens^{(q, r)}(M) \to Tens^{(q, r)}(M)$. On a local frame $(X_a)$, the anholonomy coefficients $\{ \gamma^c_{\ a b} \}$ are defined by $\gamma^c_{\ a b} := \langle e^c, [X_a, X_b] \rangle$, where $(e^a)$ is the dual local coframe. The map $\langle \cdot,\cdot \rangle : \Omega^p(M) \times \mathfrak{X}(M) \to \Omega^{p - 1}(M)$ is given by $\langle \omega, V \rangle  := \iota_V \omega$, for all $\omega \in \Omega^p(M), V \in \mathfrak{X}(M)$, where $\iota_V: \Omega^p(M) \to \Omega^{p - 1}(M)$ denotes the interior product with respect to a vector field $V$, and $\Omega^p(M)$ is the set of all (exterior differential) $p$-forms. The set of all forms is equipped with a degree-1 graded derivation, called the exterior derivative $d: \Omega^p(M) \to \Omega^{p + 1}(M)$:
\begin{align} 
d \omega (V_1, \ldots, V_{p+1}) := & \sum_{1 \leq i \leq p+1} (-1)^{i+1} V_i \left( \omega(V_1, \ldots, \check{V}_i, \ldots, V_{p+1}) \right) \nonumber\\
& + \sum_{1 \leq i < j \leq p+1} (-1)^{i+j} \omega \left( [V_i, V_j], V_1, \ldots, \check{V}_i, \ldots, \check{V}_j, \ldots, V_{p+1} \right),
\label{eb1}
\end{align}
where $\check{V}_i$ indicates that $V_i$ is excluded. The square of the exterior derivative vanishes, which follows from the Jacobi identity for the Lie bracket. The exterior derivative, interior product and the action of the Lie derivative on $p$-forms are related by the Cartan magic formula:
\begin{equation} \mathcal{L}_V = d \iota_V + \iota_V d.
\label{eb2}
\end{equation}

Metric-affine geometries are defined by a triplet $(M, g, \nabla)$ where $M$ is a manifold, $g$ is a metric, and $\nabla$ is an affine connection on $M$. The torsion and curvature operators of an affine connection $\nabla$ are defined as
\begin{align} 
& T(\nabla)(U, V) := \nabla_U V - \nabla_V U - [U, V], \nonumber\\
& R(\nabla)(U, V) W := \nabla_U \nabla_V W - \nabla_V \nabla_U W - \nabla_{[U, V]} W,
\label{eb3}
\end{align}
for all $U, V, W \in \mathfrak{X}(M)$. Moreover, for a given metric $g$ and an affine connection $\nabla$, their corresponding non-metricity tensor is defined by
\begin{equation} 
Q(\nabla, g) := \nabla g.
\label{eb4}
\end{equation}
On a local frame $(X_a)$, the torsion, curvature and non-metricity components read
\begin{align} 
&T(\nabla)^a_{\ b c} = \Gamma(\nabla)^a_{\ b c} - \Gamma(\nabla)^a_{\ c b} - \gamma^a_{\ b c}, \nonumber\\
&R(\nabla)^a_{\ b c d} = \ X_b \left( \Gamma(\nabla)^a_{\ c d} \right) - X_c \left( \Gamma(\nabla)^a_{\ b d} \right) + \Gamma(\nabla)^e_{\ c d} \Gamma(\nabla)^a_{\ b e} - \Gamma(\nabla)^e_{\ b d} \Gamma(\nabla)^a_{\ c e} - \gamma^e_{\ b c} \Gamma(\nabla)^a_{\ e d}, \nonumber\\
&Q(\nabla, g)_{a b c} = X_a(g_{b c}) - \Gamma(\nabla)^d_{\ a b} g_{d c} - \Gamma(\nabla)^d_{\ a c} g_{b d},
\label{eb5}
\end{align}
where the connection coefficients are defined as
\begin{equation} 
\Gamma(\nabla)^a_{\ b c} := \langle e^a, \nabla_{X_b} X_c \rangle.
\label{eb6}
\end{equation}
Torsion and curvature operators satisfy the first and second algebraic Bianchi identities:
\begin{align} 
&R(\nabla)(U, V)W + cycl. = (\nabla_U T(\nabla))(V, W) + T(\nabla)(T(\nabla)(U, V), W) + cycl., \nonumber\\
&(\nabla_U R(\nabla))(V, W) W' + cycl. = R(\nabla)(U, T(\nabla)(V, W)) W' + cycl.,
\label{eb7}
\end{align}
where $cycl.$ means the addition of cyclic permutations in $U, V, W$. Here, the action of the connection is extended to the whole tensor algebra by the Leibniz rule. They also satisfy the Ricci identity
\begin{equation} \nabla^2_{U, V} W - \nabla^2_{V, U} W = R(\nabla)(U, V)W - \nabla_{T(\nabla)(U, V)} W,
\label{eb8}
\end{equation}
where the second order covariant derivative is given as
\begin{equation} \nabla^2_{U, V} W := \nabla_U \nabla_V W - \nabla_{\nabla_U V} W.
\label{eb9}
\end{equation}
The Ricci identity (\ref{eb8}) is often considered as the definition of the curvature tensor for torsion-free connections.

The torsion and curvature operators satisfy the following anti-symmetry rules
\begin{align}
&T(\nabla)(U, V) = - T(\nabla)(V, U), \nonumber\\
&R(\nabla)(U, V)W = - R(\nabla)(V, U)W,
\label{eb11}
\end{align}
for all $U, V, W \in \mathfrak{X}(M)$, so that one can define the torsion and curvature 2-forms as
\begin{align}
&T^a(\nabla)(U, V) := \langle e^a, T(\nabla)(U, V) \rangle, \nonumber\\
&R^a_{\ b}(\nabla)(U, V) := \langle e^a, R(\nabla)(U, V) X_b \rangle.
\label{eb12}
\end{align}
The non-metricity and connection 1-forms are defined by
\begin{align}
&Q_{a b}(\nabla, g)(V) := Q(\nabla, g)(V, X_a, X_b), \nonumber\\
&\omega^a_{\ b}(\nabla)(V) := \langle e^a, \nabla_V X_b \rangle.
\label{eb13}
\end{align}
On the other hand the components of the forms and tensors are related by 
\begin{equation}
T^a(\nabla)_{b c} = T(\nabla)^a_{\ b c}, \qquad R^a_{\ b}(\nabla)_{c d} = R(\nabla)^a_{\ c d b}, \qquad Q_{a b}(\nabla, g)_c = Q(\nabla, g)_{c a b}.
\label{eb14}
\end{equation}
Similarly, the connection coefficients and connection 1-form components are related by
\begin{align}
\omega^a_{\ b}(\nabla)_c = \Gamma(\nabla)^a_{\ c b}.
\label{eb15}
\end{align}

By making use of these forms, the metric-affine geometries can be conveniently described in the language of forms. For example, the definitions of torsion and curvature operators become the first and second Cartan structure equations, respectively:
\begin{align} 
&T^a(\nabla) = d e^a + \omega^a_{\ b}(\nabla) \wedge e^b, \nonumber\\
&R^a_{\ b}(\nabla) = d \omega^a_{\ b}(\nabla) + \omega^a_{\ c}(\nabla) \wedge \omega^c_{\ b}(\nabla).
\label{eb16}
\end{align}
\noindent Similarly, the algebraic Bianchi identites (\ref{eb7}) lead to the differential Bianchi identities
\begin{align} 
&d T^a(\nabla) + \omega^a_{\ b}(\nabla) \wedge T^b(\nabla) = R^a_{\ b}(\nabla) \wedge e^b, \nonumber\\
&d R^a_{\ b}(\nabla) + \omega^a_{\ c}(\nabla) \wedge R^c_{\ b}(\nabla) = R^a_{\ c}(\nabla) \wedge \omega^c_{\ b}(\nabla),
\label{eb17}
\end{align}
which can be directly proven by evaluating the exterior derivative of the Cartan structure equations.

According to the fundamental theorem of semi-Riemannian geometry, for every metric $g$, there is a unique metric-$g$-compatible and torsion-free affine connection $^g \nabla$, which is called the Levi-Civita connection corresponding to $g$. It can be defined via the Koszul formula
\begin{align} 
2 g(^g \nabla_U V, W) = & \ U(g(V, W)) + V(g(U, W)) - W(g(U, V)) \nonumber\\
												& - g([V, W], U) - g([U, W], V) + g([U, V], W).
\label{eb18}
\end{align}
On a local frame, the Levi-Civita connection coefficients can be evaluated as
\begin{equation} \Gamma(^g \nabla)^a_{\ b c} = \frac{1}{2} g^{a d} \left[ X_b \left( g_{c d} \right) + X_c \left( g_{b d} \right) - X_d \left( g_{b c} \right) - \gamma^e_{\ c d} g_{e b} - \gamma^e_{\ b d} g_{e c}  + \gamma^e_{\ b c} g_{e d} \right].
\label{eb20}
\end{equation}
Moreover, the metric $g$, torsion $T(\nabla)$ and non-metricity $Q(\nabla, g)$ tensors are necessary and sufficient to determine uniquely the affine connection $\nabla$ given by:
\begin{align} 
2 g(\nabla_U V, W) = & \ 2 g(^g \nabla_U V, W) \nonumber \\
& - Q(\nabla, g)(U, V, W) - Q(\nabla, g)(V, U, W) + Q(\nabla, g)(W, U, V)  \nonumber\\
& - g(T(\nabla)(V, W), U) - g(T(\nabla)(U, W), V) + g(T(\nabla)(U, V), W).
\label{eb21}
\end{align}
On a local frame, this expression can be written as the following decomposition of the connection coefficients \cite{14}:
\begin{align} 
\Gamma(\nabla)^a_{\ b c} = & \Gamma(^g \nabla)^a_{\ b c} + \frac{1}{2} g^{a d} \Big[- Q(\nabla, g)_{b d c} + Q(\nabla, g)_{d c b} - Q(\nabla, g)_{c b d} \nonumber\\
													 & \qquad \qquad \qquad \quad \ \ - g_{e c} T(\nabla)^e_{\ b d} + g_{e d} T(\nabla)^e_{\ b c} - g_{e b} T(\nabla)^e_{\ c d} \Big].
\label{eb22} 
\end{align}


\section{Metric-Connection Geometry on Anti-Commutable Pre-Leibniz Algebroids}

\noindent In this section, geometrical objects defined usually on the tangent bundle will be generalized to vector bundles that come equipped with some necessary structures. As one can take their dual and consider the tensor products; tensors, vector fields and $p$-forms are easily defined on an arbitrary (constant rank) real vector bundle $E$ over a manifold $M$. For example, $(q, r)$-type $E$-tensors are defined as the elements of

\begin{equation}
Tens^{(q, r)}(E) := \Gamma \left( \bigotimes_{i = 1}^q E \otimes \bigotimes_{j = 1}^r E^* \right),
\label{ec1}
\end{equation}
whereas the sets of $E$-vector fields and $E$-$p$-forms are denoted by $\mathfrak{X}(E)$ and $\Omega^p(E)$, respectively. $E$-interior product $\iota_V$ of an $E$-$p$-form with respect to an $E$-vector field can be defined analogously and it is a degree-$(-1)$ graded derivation on the set of all $E$-forms. With these definitions, a fiber-wise metric on $E$ becomes a non-degenerate, symmetric $(0, 2)$-type $E$-tensor, which is also called an $E$-metric. If $(E, \rho)$ is an anchored vector bundle, i. e. $\rho: E \to T(M)$ is a vector bundle morphism over the identity, then one can define an $E$-connection on a vector bundle $R$ over $M$ as an $\mathbb{R}$-bilinear map $\nabla: \mathfrak{X}(E) \times \mathfrak{X}(R) \to \mathfrak{X}(R), (v, r) \mapsto \nabla_v r$
\begin{align} 
&\nabla_v (f r) = \rho(v)(f) r + f \nabla_v r, \nonumber\\
&\nabla_{f v} r = f \nabla_v r,
\label{ec2}
\end{align}
for all $v \in \mathfrak{X}(E), r \in \mathfrak{X}(R), f \in C^{\infty}(M, \mathbb{R})$ \cite{15}. It can also be considered as a map $\mathfrak{X}(R) \to \Omega^1(E) \times \mathfrak{X}(R)$. Note that, the anchor $\rho: E \to T(M)$ induces a map between the sections as it is a vector bundle morphism over the identity, and this new map is also denoted by the same letter; $\rho: \mathfrak{X}(E) \to \mathfrak{X}(M)$. An $E$-connection on $E$ itself is called a linear $E$-connection. The action of a linear $E$-connection can be extended to the whole $E$-tensor algebra by the Leibniz rule, similar to the usual case. On a local $E$-frame $(X_a)$, $E$-connection coefficients are defined by
\begin{equation} \Gamma(\nabla)^a_{\ b c} := \langle e^a, \nabla_{X_b} X_c \rangle, 
\label{ec3}
\end{equation}
where the map $\langle \cdot,\cdot \rangle : \Omega^p(E) \times \mathfrak{X}(E) \to \Omega^{p - 1}(E)$ is defined by $\langle \Omega, v \rangle  := \iota_v(\Omega)$, for all $\Omega \in \Omega^p(E), v \in \mathfrak{X}(E)$. Given an $E$-metric $g$ and a linear $E$-connection $\nabla$, one can define the $E$-non-metricity tensor $Q(\nabla, g) := \nabla g$. If $Q(\nabla, g) = 0$, then $\nabla$ is called $E$-metric-$g$-compatible.

In order to continue with the definitions of $E$-torsion and $E$-curvature operators, one needs to introduce a bracket.

\begin{Definition} 
A triplet $(E, \rho, [\cdot,\cdot]_E)$ is called an almost-Leibniz algebroid if $(E, \rho)$ is an anchored vector bundle over $M$, $[\cdot,\cdot]_E: \mathfrak{X}(E) \times \mathfrak{X}(E) \to \mathfrak{X}(E)$ is an $\mathbb{R}$-bilinear map that satisfies the right-Leibniz rule
\begin{equation} 
[u, f v]_E = \rho(u)(f) v + f [u, v]_E,
\label{ec4}
\end{equation}
for all $u, v \in \mathfrak{X}(E), f \in C^{\infty}(M, \mathbb{R})$.
\label{dc1}
\end{Definition}
\noindent $E$-anholonomy coefficients can be defined as $\gamma^a_{\ b c} := \langle e^a, [X_b, X_c]_E \rangle$ on a local $E$-frame. 

\begin{Definition} 
A quadruplet $(E, \rho, [\cdot,\cdot]_E, L)$ is called a local almost-Leibniz algebroid over $M$ if $(E, \rho, [\cdot,\cdot]_E)$ is an almost-Leibniz algebroid, and $L: \Omega^1(E) \times \mathfrak{X}(E) \times \mathfrak{X}(E) \to \mathfrak{X}(E)$ is a $C^{\infty}(M, \mathbb{R})$-multilinear map, called the locality operator, that satisfies the left-Leibniz rule \cite{16}, \cite{13}
\begin{equation} [f u, v]_E = - \rho(v)(f) u + f [u, v]_E + L(Df, u, v),
\label{ec5}
\end{equation}
for all $u, v \in \mathfrak{X}(E), f \in C^{\infty}(M, \mathbb{R})$, where the coboundary map $D: C^{\infty}(M, \mathbb{R}) \to \Omega^1(E)$ is defined by $(Df)(u) := \rho(u)(f)$ \cite{17}.
\label{dc2}
\end{Definition}

\begin{Proposition} Let $(E, \rho, [\cdot,\cdot]_E)$ be an almost-Leibniz algebroid over $M$ and $L: \Omega^1(E) \times \mathfrak{X}(E) \times \mathfrak{X}(E) \to \mathfrak{X}(E)$ a $C^{\infty}(M, \mathbb{R})$-multilinear map. Then, the following condition is sufficient for $(E, \rho, [\cdot,\cdot]_E, L)$ to be a local almost-Leibniz algebroid:
\begin{equation} [u, v]_E + [v, u]_E = S(u, v),
\label{ec6}
\end{equation}
for all $u, v \in \mathfrak{X}(E)$, where $S: \mathfrak{X}(E) \times \mathfrak{X}(E) \to \mathfrak{X}(E)$ is a map that satisfies
\begin{equation} S(f u, v) = f S(u, v) + L(Df, u, v),
\label{ec7}
\end{equation}
for all $f \in C^{\infty}(M, \mathbb{R}), u, v \in \mathfrak{X}(E)$.
\label{pc1}
\end{Proposition}

\begin{proof} We need to show that the equations (\ref{ec6}, \ref{ec7}) and (\ref{ec4}) are enough to prove the equation (\ref{ec5}):
\begin{align} 
[f u, v]_E &= S(f u, v) - [v, f u]_E \nonumber\\
&= f S(u, v) + L(Df, u, v) - \left\{ \rho(v)(f) u + f [v, u]_E \right\} \nonumber\\
&= f S(u, v) + L(Df, u, v) - \rho(v)(f) u - f \left\{ S(u, v) - [u, v]_E \right\} \nonumber\\
&= - \rho(v)(f) u + f [u, v]_E + L(Df, u, v). \nonumber
\end{align}
\end{proof}

\begin{Corollary} Let $(E, \rho, [\cdot,\cdot]_E)$ be an almost-Leibniz algebroid whose bracket satisfies
\begin{equation} [u, v]_E + [v, u]_E = L(e^a, \nabla_{X_a} u, v) + L(e^a, \nabla_{X_a} v, u),
\label{ec8}
\end{equation}
for some linear $E$-connection $\nabla$ and some $C^{\infty}(M, \mathbb{R})$-multilinear map $L: \Omega^1(E) \times \mathfrak{X}(E) \times \mathfrak{X}(E) \to \mathfrak{X}(E)$, where $(X_a)$ is a local $E$-frame, and $(e^a)$ is its dual local $E$-coframe. Then, $(E, \rho, [\cdot,\cdot]_E, L)$ is a local almost-Leibniz algebroid.
\label{cc1}
\end{Corollary}

\begin{proof} We need to show that the right-hand side of the equation (\ref{ec8}) satisfies the equation (\ref{ec7}). Let us choose $S(u, v) := L(e^a, \nabla_{X_a} u, v) + L(e^a, \nabla_{X_a} v, u)$
\begin{align} 
S(f u, v) &= L(e^a, \nabla_{X_a} (f u), v) + L(e^a, \nabla_{X_a} v, f u) \nonumber\\
&= L(e^a, \rho(X_a)(f) u + f \nabla_{X_a} u, v) + f L(e^a, \nabla_{X_a} v, u) \nonumber\\
&= L(\rho(X_a)(f) e^a, u, v) + f L(e^a, \nabla_{X_a} u, v) + f L(e^a, \nabla_{X_a} v, u) \nonumber\\
&= L(Df, u, v) + f \left\{ L(e^a, \nabla_{X_a} u, v) + L(e^a, \nabla_{X_a} v, u) \right\} \nonumber\\
&= f S(u, v) + L(Df, u, v), \nonumber
\end{align}
where we use the fact that $\rho(X_a)(f) e^a = Df$ and the definition of a linear $E$-connection.
\end{proof}

\noindent Note that the choice $S(u, v) = L(e^a, \nabla_{X_a} u, v)$ also satisfies the condition (\ref{ec7}), but $C^{\infty}(M, \mathbb{R})$-multilinearity of $L$ and the symmetry of $S$ dictate the trivial case $L = 0$. 

\begin{Definition} The quintet $(E, \rho, [\cdot,\cdot]_E, L, \nabla)$ is called an anti-commutable almost-Leibniz algebroid if $(E, \rho, [\cdot,\cdot]_E)$ is an almost-Leibniz algebroid, $L: \Omega^1(E) \times \mathfrak{X}(E) \times \mathfrak{X}(E) \to \mathfrak{X}(E)$ is a $C^{\infty}(M, \mathbb{R})$-multilinear map, $\nabla$ is a linear $E$-connection, and the condition (\ref{ec8}):
\begin{equation} [u, v]_E + [v, u]_E = L(e^a, \nabla_{X_a} u, v) + L(e^a, \nabla_{X_a} v, u),
\nonumber
\end{equation}
holds for all $u, v \in \mathfrak{X}(E)$.
\label{dc3}
\end{Definition}

If $(E, \rho, [\cdot,\cdot]_E, L, \nabla)$ is an anti-commutable almost-Leibniz algebroid, so is $(E, \rho, [\cdot,\cdot]_E, L, \nabla')$ for any $\nabla'$ that satisfies
\begin{equation} L(e^a, \Delta(\nabla, \nabla')(X_a, u), v)) = - L(e^a, \Delta(\nabla, \nabla')(X_a, v), u)),
\label{ec9}
\end{equation}
for all $u, v \in \mathfrak{X}(E)$, where $\Delta(\nabla, \nabla')$ is the difference $E$-tensor between the linear $E$-connections $\nabla$ and $\nabla'$, which is defined by
\begin{equation} \Delta(\nabla, \nabla')(u, v) := \nabla_u v - \nabla'_u v,
\label{ec10}
\end{equation}
for all $u, v \in \mathfrak{X}(E)$. The condition (\ref{ec9}) is an equivalence relation in the affine space of linear $E$-connections. The equivalence class of a linear $E$-connection $\nabla$ will be denoted by $[\nabla]_L$. For an almost-Leibniz algebroid $(E, \rho, [\cdot,\cdot]_E)$, if one is given a multilinear map $L$, then this equivalence class $[\nabla]_L$ can be evaluated. Similarly, given a linear $E$-connection $\nabla$, one can evaluate an equivalence class $[L]_{\nabla}$ of multilinear maps. Or, the most general case would be the one with the equivalence class $[(L, \nabla)]$, where neither $L$ and $\nabla$ are fixed. The most natural case seems to be the one in which the multilinear map $L$ is fixed, and the equivalence class $[\nabla]_L$ of linear $E$-connections is evaluated accordingly. This leads us to the following definition:
\begin{Definition} On an anti-commutable almost-Leibniz algebroid, linear $E$-connections in the equivalence class $[\nabla]_L$ are called admissible.
\label{dc4}
\end{Definition}

Main motivation behind these admissible linear $E$-connections comes from the torsion and curvature operators. The naive generalizations of the torsion and curvature operators (\ref{eb3}) do not work for local almost-Leibniz algebroids, as the $C^{\infty}(M, \mathbb{R})$-multilinearity does not hold so that they are not tensorial.
In \cite{12}, we proved that there is an affine map $^A \mathfrak{T}: \nabla \mapsto \ ^A [\cdot,\cdot]_{\nabla}$ from the affine space of linear $E$-connections to local almost-Leibniz brackets on $E$ defined by
\begin{equation} 
^A [u, v]_{\nabla} := \nabla_u v - \nabla_v u + A(u, v),
\label{ec11}
\end{equation}
for any $\mathbb{R}$-bilinear map $A: \mathfrak{X}(E) \times \mathfrak{X}(E) \to \mathfrak{X}(E)$ satisfying 
\begin{align} 
A(f u, v) &= L(Df, u, v) + f A(u, v), \nonumber\\
A(u, f v) &= f A(u, v),
\label{ec12}
\end{align}
for all $u, v \in \mathfrak{X}(E), f \in C^{\infty}(M, \mathbb{R})$. Trivially, this also implies that adding the $A$ term to the definition of the $E$-torsion operator would make it tensorial. For a local almost-Leibniz algebroid, the $E$-torsion operator suggested in \cite{13} is given by
\begin{equation} 
T(\nabla)(u, v) := \nabla_u v - \nabla_v u - [u, v]_E + L(e^a, \nabla_{X_a} u, v).
\label{ec13}
\end{equation}
Note that the extra term $L(e^a, \nabla_{X_a} u, v)$ satisfies the necessary conditions (\ref{ec12}), but the geometric meaning behind this modification was missing. We claim that the use of admissible linear $E$-connections might clarify the geometric interpretation: One should work with an admissible linear $E$-connection when trying to construct a metric-connection geometry on an algebroid structure. The choice of an $E$-connection changes the choice of the bracket so that these two structures are not completely independent from each other. This interpretation can be seen as follows. The modification of the $E$-torsion operator can be considered as a modification of the bracket $[\cdot,\cdot]_E$:
\begin{equation} [u, v]_E^{\nabla} := [u, v]_E - L(e^a, \nabla_{X_a} u, v).
\label{ec14}
\end{equation}
This new bracket will be called ``\textit{the modified bracket}'', and in terms of it, the $E$-torsion operator (\ref{ec13}) becomes
\begin{equation} T(\nabla)(u, v) = \nabla_u v - \nabla_v u - [u, v]_E^{\nabla}.
\label{ec15}
\end{equation}
One can analogously define ``\textit{modified anholonomy coefficients}'' as 
\begin{equation} \gamma(\nabla)^a_{\ b c} := \langle e^a, [X_b, X_c]_E^{\nabla} \rangle = \gamma^a_{\ b c} - \Gamma(\nabla)^e_{\ d b} L^{a d}_{\ \ e c},
\label{ec16}
\end{equation}
so that the $E$-torsion coefficients can be written as
\begin{align} 
T(\nabla)^a_{\ b c} &= \Gamma(\nabla)^a_{\ b c} - \Gamma(\nabla)^a_{\ c b} - \gamma^a_{\ b c} + \Gamma(\nabla)^e_{\ d b} L^{a d}_{\ \ e c}, \nonumber\\
&= \Gamma(\nabla)^a_{\ b c} - \Gamma(\nabla)^a_{\ c b} - \gamma(\nabla)^a_{\ b c}.
\label{ec17}
\end{align}
One can observe that every local almost-Leibniz bracket in the range of the map $^A \mathfrak{T}$ is an anti-commutable bracket when $A$ is chosen as $L(e^a, \nabla_{X_a} u, v)$:
\begin{align} 
^A [u, v]_{\nabla} + \ ^A [v, u]_{\nabla} &= \left[ \nabla_u v - \nabla_v u + A(u, v) \right] + \left[ \nabla_v u - \nabla_u v + A(v, u) \right] \nonumber\\
&= A(u, v) + A(v, u) \nonumber\\
&= L(e^a, \nabla_{X_a} u, v) + L(e^a, \nabla_{X_a} v, u)
\label{ec18}
\end{align}
More generally, the maps $S$ and $A$ are directly related by $S(u, v) = A(u, v) + A(v, u)$.

The modified bracket makes it clear the following simple, yet important, fact:

\begin{Proposition} For an almost-Leibniz algebroid, an $E$-torsion-free linear $E$-connection $\nabla$ has to be admissible.
\label{pc1b}
\end{Proposition}

\begin{proof} In terms of the modified bracket the $E$-torsion operator of a linear $E$-connection $\nabla$ is written as (\ref{ec15}):
\begin{equation} T(\nabla)(u, v) = \nabla_u v - \nabla_v u - [u, v]_E^{\nabla}, \nonumber
\end{equation}
which implies for the $E$-torsion-free case:
\begin{equation} [u, v]_E^{\nabla} = \nabla_u v - \nabla_v u. \nonumber
\end{equation}
As the right-hand side is anti-symmetric in $u$ and $v$, the modified bracket has to be anti-symmetric. Hence, $\nabla$ is admissible.
\end{proof}

The modified bracket (\ref{ec14}) is related with different algebroid structures. Inspired from the dull algebroid definition in \cite{9}, we define almost-dull algebroids:

\begin{Definition} The triplet $(E, \rho, [\cdot,\cdot]_E)$ is called an almost-dull algebroid, if $(E, \rho)$ is an anchored vector bundle and the bracket satisfies
\begin{equation} [f_1 u, f_2 v]_E = f_1 f_2 [u, v]_E + f_1 \rho(u)(f_2) v - f_2 \rho(v)(f_1) u,
\label{ec19}
\end{equation}
for all $f_1, f_2 \in C^{\infty}(M, \mathbb{R}), u, v \in \mathfrak{X}(E)$.
\label{dc5}
\end{Definition}
\noindent Note that, this property (\ref{ec19}) is equivalent to the fact that an almost-dull algebroid is a local almost-Leibniz algebroid whose locality operator vanishes. 

\begin{Definition} An almost-Leibniz algebroid $(E, \rho, [\cdot,\cdot]_E)$ is called an almost-Lie algebroid if the bracket $[\cdot,\cdot]_E$ is anti-symmetric.
\label{dc6}
\end{Definition}
\noindent Every almost-Lie algebroid is an almost-dull algebroid due to the anti-symmetry of the bracket. Moreover, almost-dull algebroids that are not almost-Lie constitute examples of almost-Leibniz algebroids which are not anti-commutable.

\begin{Proposition} Given some $L$, if $(E, \rho, [\cdot,\cdot]_E)$ is an almost-Leibniz algebroid, then $(E, \rho, [\cdot,\cdot]_E^{\nabla})$ is an almost-dull algebroid. Moreover, if $\nabla$ is an admissible linear $E$-connection, then $(E, \rho, [\cdot,\cdot]_E^{\nabla})$ is an almost-Lie algebroid.
\label{pc2}
\end{Proposition}

\begin{proof} We need to show that the modified bracket $[.,.]_E^{\nabla}$ satisfies the condition (\ref{ec19}). Equivalently, we can show that $[.,.]_E^{\nabla}$ is a local almost-Leibniz bracket with a vanishing locality operator. First, we prove the right-Leibniz rule:
\begin{align} 
[u, f v]_E^{\nabla} &= [u, f v]_E - L(e^a, \nabla_{X_a} u, f v) \nonumber\\
&= \rho(u)(f) v + f [u, v]_E - f L(e^a, \nabla_{X_a} u, v) \nonumber\\
&= \rho(u)(f) v + f \left\{ [u, v]_E - L(e^a, \nabla_{X_a} u, v) \right\} \nonumber\\
&= \rho(u)(f) v + f [u, v]_E^{\nabla}, \nonumber
\end{align}
so that $(E, \rho, [.,.]_E^{\nabla})$ is an almost-Leibniz algebroid. Next, we prove the left-Leibniz rule:
\begin{align} 
[f u, v]_E^{\nabla} &= [f u, v]_E - L(e^a, \nabla_{X_a} (f u), v) \nonumber\\
&= - \rho(v)(f) u + f [u, v]_E + L(Df, u, v) - L(e^a, \rho(X_a)(f) u + f \nabla_{X_a} u, v) \nonumber\\
&= - \rho(v)(f) u + f [u, v]_E + L(Df, u, v) - L(Df, u, v) - f L(e^a, \nabla_{X_a} u, v) \nonumber\\
&= - \rho(v)(f) u + f \left\{ [u, v]_E - L(e^a, \nabla_{X_a} u, v) \right\} \nonumber\\
&= - \rho(v)(f) u + f [u, v]_E^{\nabla} + 0. \nonumber
\end{align}
This means that the locality operator on the almost-Leibniz algebroid $(E, \rho, [.,.]_E^{\nabla})$ can be chosen as 0. Hence, we proved that $(E, \rho, [.,.]_E^{\nabla})$ is an almost-dull algebroid. Moreover, if $\nabla$ is admissible, then 
\begin{equation} 
[v, u]_E^{\nabla} = [v, u]_E - L(e^a, \nabla_{X_a} v, u) = - [u, v]_E + L(e^a, \nabla_{X_a} u, v) = - [u, v]_E^{\nabla} \nonumber
\end{equation}
where we used the definition (\ref{ec8}) of an admissible linear $E$-connection. This means that the modified bracket $[.,.]_E^{\nabla}$ is anti-symmetric, so that $(E, \rho, [.,.]_E^{\nabla})$ is an almost-Lie algebroid.
\end{proof}

\noindent The fact that $[\cdot,\cdot]_E^{\nabla}$ is anti-symmetric for an admissible linear $E$-connection $\nabla$ is the reason behind the adjective ``\textit{anti-commutable}''. In some sense, one can make the bracket anti-commutative (or anti-symmetric). One natural question is that which of these modified brackets $[.,.]_E^{\nabla}$ correspond to the anti-symmetrization of the original bracket $[.,.]_E$. This is the case when an admissible linear $E$-connection $\nabla'$ satisfies $L(e^a, \nabla'_{X_a} u, v) = L(e^a, \nabla'_{X_a} v, u)$. In this case, the anholonomy coefficients decomposes into its anti-symmetric and symmetric parts respectively as:
\begin{equation} \gamma^a_{\ b c} = \gamma(\nabla')^a_{\ b c} + \Gamma(\nabla')^d_{\ a b} L^{e a}_{\ \ d c}.
\label{ec19b}
\end{equation}

\begin{Corollary} For an admissible linear $E$-connection $\nabla$, the $E$-torsion operator $T(\nabla)$ is anti-symmetric:
\begin{equation} T(\nabla)(u, v) = - T(\nabla)(v, u),
\label{ec20}
\end{equation}
for all $u, v \in \mathfrak{X}(E)$.
\label{cc2}
\end{Corollary}

\begin{proof} This follows from the fact that the modified bracket $[.,.]_E^{\nabla}$ is anti-symmetric for an admissible linear $E$-connection $\nabla$. The $E$-torsion operator $T(\nabla)$ can be seen as the $E$-torsion operator on the almost-Lie algebroid $(E, \rho, [.,.]_E^{\nabla})$ without any modification.
\end{proof}

In order to define an $E$-curvature operator, one needs to assume that the anchor respects the brackets.

\begin{Definition} An almost-Leibniz algebroid $(E, \rho,$ $[\cdot,\cdot]_E)$ is called a pre-Leibniz algebroid if 
\begin{equation}
\rho([u, v]_E) = [\rho(u), \rho(v)],
\label{ec21}
\end{equation}
for all $u, v \in \mathfrak{X}(E)$.
\label{dc7}
\end{Definition}
\noindent Moreover, to achieve tensorial curvature one needs to introduce a locality projector, which we defined in \cite{12}. As the locality projectors are related to the kernel of the anchor, one needs to assume that the anchored vector bundle is regular, i. e. the kernel of the anchor is of locally constant rank, so that the kernel defines a subbundle.
\begin{Definition} On a regular local almost-Leibniz algebroid $(E, \rho, [\cdot,\cdot]_E, L)$ , a locality projector is defined as a $C^{\infty}(M, \mathbb{R})$-linear map $\mathcal{P}: \mathfrak{X}(E) \to \mathfrak{X}(E)$ such that
\begin{enumerate}
\item The image of the projected locality operator $\hat{L} := \mathcal{P} L$ is a subset of the kernel of the anchor, i.e. $im(\hat{L}) \subset ker(\rho)$,
\item The restriction of the locality projector on the kernel of the anchor is the identity map, i.e. $\mathcal{P}|_{ker(\rho)} = id_{ker(\rho)}$.
\end{enumerate}
\label{dc8}
\end{Definition}
\noindent With the help of a locality projector, one can define the $E$-curvature operator.

\begin{Definition}
Given a locality projector $\mathcal{P}$ on $(E, \rho, [\cdot,\cdot]_E, L)$, the $E$-curvature operator of a linear $E$-connection $\nabla$ is defined by
\begin{equation}
R(\nabla)(u, v) w := \nabla_u \nabla_v w - \nabla_v \nabla_u w - \nabla_{[u, v]_E} w + \nabla_{\hat{L}(e^a, \nabla_{X_a} u, v)} w.
\label{ec22}
\end{equation}
\label{dc9}
\end{Definition}
\noindent Similar to the $E$-torsion operator, one can modify the bracket in order to express the $E$-curvature operator more naturally on an algebroid structure. In order to understand what type of bracket will be the result of this new modification, we need pre-dull algebroids.

\begin{Definition} An almost-dull algebroid $(E, \rho, [\cdot,\cdot]_E)$ is called a pre-dull algebroid if the equation (\ref{ec21}) is satisfied \cite{9}.
\label{dc10}
\end{Definition}
\noindent Note that such an algebroid is called a ``dull algebroid'' in the reference \cite{9}, but it is more appropriate to call it pre-dull for the purposes of this paper. Moreover, a general caution is in place; in the literature many algebroid structures are defined by using different names, so one should be careful about the assumptions in a specific article. 

\begin{Definition} A pre-Leibniz algebroid $(E, \rho, [\cdot,\cdot]_E)$ is called a pre-Lie algebroid if the bracket $[\cdot,\cdot]_E$ is anti-symmetric.
\label{dc11}
\end{Definition}
\noindent Every pre-Lie algebroid is automatically a pre-dull algebroid due to the anti-symmetry of the bracket. Note that the $E$-curvature operator (\ref{ec22}) can be written by using the modified bracket (\ref{ec14}) as follows
\begin{equation}
R(\nabla)(u, v) w := \nabla_u \nabla_v w - \nabla_v \nabla_u w - \nabla_{[u, v]_E^{\nabla}} w - \nabla_{(1 - \mathcal{P}) L(e^a, \nabla_{X_a} u, v)} w.
\label{ec23}
\end{equation}
\noindent Hence, one can see that it is not the same as the $E$-curvature operator on the pre-Lie algebroid $(E, \rho, [\cdot,\cdot]_E^{\nabla})$. Actually, $(E, \rho, [\cdot,\cdot]_E^{\nabla})$ is not even a pre-Lie algebroid in the most general case. It is a pre-Lie algebroid only if the image of the locality projector is already a subset of the kernel of the anchor; in this case one does not need the locality projector. However, one can have another modification of the bracket $[\cdot,\cdot]_E$ by using the locality projector, which will be called ``\textit{the projected modified bracket}'':

\begin{equation} [u, v]_E^{\hat{\nabla}} := [u, v]_E - \hat{L}(e^a, \nabla_{X_a} u, v) = [u, v]_E - \mathcal{P} L(e^a, \nabla_{X_a} u, v).
\label{ec24}
\end{equation}
\noindent By using this projected modified bracket, the $E$-curvature operator (\ref{ec22}) can be written as

\begin{equation} R(\nabla)(u, v) w := \nabla_u \nabla_v w - \nabla_v \nabla_u w - \nabla_{[u, v]_E^{\hat{\nabla}}} w.
\label{ec25}
\end{equation}
Similarly, one can define ``\textit{projected modified anholonomy coefficients}'' as 
\begin{equation} \hat{\gamma}(\nabla)^a_{\ b c} := \langle e^a, [X_b, X_c]_E^{\hat{\nabla}} \rangle = \gamma^a_{\ b c} - \Gamma(\nabla)^e_{\ d b} \hat{L}^{a d}_{\ \ e c} = \gamma^a_{\ b c} - \Gamma(\nabla)^e_{\ d b} L^{a f}_{\ \ e c} \mathcal{P}^d_{\ f},
\label{ec26}
\end{equation}
so that $E$-curvature coefficients can be written as
\begin{align} 
R(\nabla)^a_{\ b c d} = & \ \rho(X_b) \left( \Gamma(\nabla)^a_{\ c d} \right) - \rho(X_c) \left( \Gamma(\nabla)^a_{\ b d} \right) + \Gamma(\nabla)^e_{\ c d} \Gamma(\nabla)^a_{\ b e} \nonumber\\
& - \Gamma(\nabla)^e_{\ b d} \Gamma(\nabla)^a_{\ c e} - \gamma^e_{\ b c} \Gamma(\nabla)^a_{\ e d} + \Gamma(\nabla)^a_{\ g d} \Gamma(\nabla)^f_{\ e b} \hat{L}^{g e}_{\ \ f c}, \nonumber\\
= & \ \rho(X_b) \left( \Gamma(\nabla)^a_{\ c d} \right) - \rho(X_c) \left( \Gamma(\nabla)^a_{\ b d} \right) + \Gamma(\nabla)^e_{\ c d} \Gamma(\nabla)^a_{\ b e} \nonumber\\
& - \Gamma(\nabla)^e_{\ b d} \Gamma(\nabla)^a_{\ c e} - \hat{\gamma}(\nabla)^e_{\ b c} \Gamma(\nabla)^a_{\ e d}
\label{ec27}
\end{align}

\begin{Proposition} Given some $L$, if $(E, \rho, [\cdot,\cdot]_E)$ is a pre-Leibniz algebroid, then $(E, \rho, [\cdot,\cdot]_E^{\hat{\nabla}})$ is a pre-dull algebroid. Moreover, if $\nabla$ is an admissible linear $E$-connection, then $(E, \rho, [\cdot,\cdot]_E^{\hat{\nabla}})$ is a pre-Lie algebroid.
\label{pc3}
\end{Proposition}

\begin{proof} We need to show that the projected modified bracket $[.,.]_E^{\hat{\nabla}}$ is a local almost-Leibniz bracket with a vanishing locality operator. First, we prove the right-Leibniz rule:
\begin{align} 
[u, f v]_E^{\hat{\nabla}} &= [u, f v]_E - \mathcal{P} L(e^a, \nabla_{X_a} u, f v) \nonumber\\
&= \rho(u)(f) v + f [u, v]_E - \mathcal{P} (f L(e^a, \nabla_{X_a} u, v)) \nonumber\\
&= \rho(u)(f) v + f [u, v]_E - f \mathcal{P} L(e^a, \nabla_{X_a} u, v) \nonumber\\
&= \rho(u)(f) v + f \left\{ [u, v]_E - \mathcal{P} L(e^a, \nabla_{X_a} u, v) \right\} \nonumber\\
&= \rho(u)(f) v + f [u, v]_E^{\hat{\nabla}}, \nonumber
\end{align}
so that $(E, \rho, [.,.]_E^{\hat{\nabla}})$ is an almost-Leibniz algebroid. Next, we prove the left-Leibniz rule:
\begin{align} 
[f u, v]_E^{\hat{\nabla}} &= [f u, v]_E - \mathcal{P} L(e^a, \nabla_{X_a} (f u), v) \nonumber\\
&= - \rho(v)(f) u + f [u, v]_E + L(Df, u, v) - \mathcal{P} L(e^a, \rho(X_a)(f) u + f \nabla_{X_a} u, v) \nonumber\\
&= - \rho(v)(f) u + f [u, v]_E + L(Df, u, v) - \mathcal{P} L(Df, u, v) - \mathcal{P} (f L(e^a, \nabla_{X_a} u, v)) \nonumber\\
&= - \rho(v)(f) u + f [u, v]_E + L(Df, u, v) - L(Df, u, v) - f \mathcal{P} L(e^a, \nabla_{X_a} u, v) \nonumber\\
&= - \rho(v)(f) u + f \left\{ [u, v]_E - \mathcal{P} L(e^a, \nabla_{X_a} u, v) \right\} \nonumber\\
&= - \rho(v)(f) u + f [u, v]_E^{\hat{\nabla}} + 0. \nonumber
\end{align}
This means that the locality operator on the almost-Leibniz algebroid $(E, \rho, [.,.]_E^{\hat{\nabla}})$ can be chosen as 0. Hence, we proved that $(E, \rho, [.,.]_E^{\hat{\nabla}})$ is an almost-dull algebroid. Note that here we used the fact that $\rho(L(Df, u, v)) = 0$ for all $u, v \in \mathfrak{X}(E), f \in C^{\infty}(M, \mathbb{R})$, so that $\mathcal{P} L(Df, u, v) = L(Df, u, v)$. Next, we prove that the projected modified bracket is a pre-dull bracket:
\begin{align} 
\rho \left( [u, v]_E^{\hat{\nabla}} \right) &= \rho \left( [u, v]_E - \mathcal{P} L(e^a, \nabla_{X_a} u, v) \right) \nonumber\\
&= \rho([u, v]_E) - \rho(\mathcal{P} L(e^a, \nabla_{X_a} u, v)) \nonumber\\
&= \rho([u, v]_E) - 0 \nonumber\\
&= [\rho(u), \rho(v)]. \nonumber
\end{align}
Hence, $(E, \rho, [.,.]_E^{\hat{\nabla}})$ is a pre-dull algebroid. Next, we prove that the projected modified bracket is $[.,.]_E^{\hat{\nabla}}$ is anti-symmetric for an admissible linear $E$-connection $\nabla$. In order to prove this, we need to observe that the definition
\begin{equation} [u, v]_E + [v, u]_E = L(e^a, \nabla_{X_a} u, v) + L(e^a, \nabla_{X_a} v, u), \nonumber
\end{equation}
implies that
\begin{equation} \rho \left( [u, v]_E + [v, u]_E \right) = \rho \left( L(e^a, \nabla_{X_a} u, v) + L(e^a, \nabla_{X_a} v, u) \right). \nonumber
\end{equation}
As the bracket $[.,.]_E$ is a pre-Leibniz bracket, then the left-hand side of this equation vanishes:
\begin{equation} \rho \left( [u, v]_E + [v, u]_E \right) = \rho \left( [u, v]_E \right) + \rho \left( [v, u]_E \right) = [\rho(u), \rho(v)] + [\rho(v), \rho(u)] = 0, \nonumber
\end{equation}
which follows from the anti-symmetry of the usual Lie bracket. Hence, we get
\begin{equation} \rho \left( L(e^a, \nabla_{X_a} u, v) + L(e^a, \nabla_{X_a} v, u) \right) = 0, \nonumber
\end{equation}
which means that $L(e^a, \nabla_{X_a} u, v) + L(e^a, \nabla_{X_a} v, u)$ is in the kernel of the anchor $\rho$, so that 
\begin{equation} \mathcal{P} \left( L(e^a, \nabla_{X_a} u, v) + L(e^a, \nabla_{X_a} v, u) \right) = L(e^a, \nabla_{X_a} u, v) + L(e^a, \nabla_{X_a} v, u). \nonumber
\end{equation}
The $\mathbb{R}$-linearity of $\mathcal{P}$ and the admissibility of $\nabla$ imply that 
\begin{equation} [u, v]_E + [v, u]_E = \mathcal{P} L(e^a, \nabla_{X_a} u, v) + \mathcal{P} L(e^a, \nabla_{X_a} v, u). \nonumber\\
\end{equation}
In other words, $\mathcal{P} L \in [L]_{\nabla}$. Now with this information, we can prove the anti-symmetry of the projected modified bracket for an admissible linear $E$-connection:
\begin{equation} 
[v, u]_E^{\hat{\nabla}} = [v, u]_E - \mathcal{P} L(e^a, \nabla_{X_a} v, u) = - [u, v]_E - \mathcal{P} L(e^a, \nabla_{X_a} u, v)
= - [u, v]_E^{\hat{\nabla}}. \nonumber
\end{equation}
Hence we proved $(E, \rho, [.,.]_E^{\hat{\nabla}})$ is a pre-Lie algebroid for an admissible linear $E$-connection $\nabla$.
\end{proof}

\begin{Corollary} For an admissible linear $E$-connection $\nabla$, the $E$-curvature operator satisfies the following anti-symmetry property
\begin{equation} R(\nabla)(u, v) w = - R(\nabla)(v, u) w,
\label{ec28}
\end{equation}
for all $u, v, w \in \mathfrak{X}(E)$.
\label{cc3}
\end{Corollary}

\begin{proof} This follows from the fact that for an admissible linear $E$-connection, the projected modified $[.,.]_E^{\hat{\nabla}}$ is anti-symmetric. The $E$-curvature operator $R(\nabla)$ can be seen as the $E$-curvature operator on the pre-Lie algebroid $(E, \rho, [.,.]_E^{\hat{\nabla}})$ without any modification.
\end{proof}

Before going on with the detailed properties of admissible linear $E$-connections, we should focus on some examples from the literature.

\begin{itemize}
\item \textbf{Almost-Lie algebroids:} Every almost-Lie algebroid is an almost-dull algebroid, so one can consider them as local almost-Leibniz algebroids with a vanishing locality operator. As $L = 0$, every linear $E$-connection is admissible. In particular, $(T(M), id_{T(M)}, [\cdot,\cdot], 0)$ is a local almost-Leibniz algebroid. In the usual metric-affine geometry one can consider any arbitrary connection and this is coherent with the fact that every linear $E$-connection on an almost-Lie algebroid is admissible. Every \textbf{pre-Lie} and \textbf{Lie algebroid}, where the letter also satisfies the Jacobi identity, is an almost-Lie algebroid. Moreover, \textbf{$H$-twisted Lie algebroids} \cite{18} that satisfy a modification of the Jacobi identity by an $E$-3-form is also an almost-Lie algebroid. Hence, everything is also valid for these algebroids.
\item \textbf{Almost-Courant algebroids:} An almost-Leibniz algebroid $(E, \rho, [\cdot,\cdot]_E)$ whose bracket satisfies

\begin{equation} [u, v]_E + [v, u]_E = g^{-1}(D(g(u, v))),
\label{ec29}
\end{equation}
for some $E$-metric $g$ is called an almost-Courant algebroid, which is denoted by $(E, \rho, [\cdot,\cdot]_E, g)$ \cite{12}. Every \textbf{Courant algebroid} is in fact an almost-Courant algebroid, so everything is also valid for Courant algebroids, which are the necessary mathematical structure in generalized geometries. \textbf{Pre-Courant algebroids} in the sense of \cite{19}, \cite{20} and $H$\textbf{-twisted Courant algebroids} \cite{21} are also almost-Courant algebroids by definition. Moreover, every \textbf{metric algebroid} is also an almost-Courant algebroid \cite{6} which underlies the para-Hermitian formulation of double field theory.
\begin{Proposition} On an almost-Courant algebroid $(E, \rho, [\cdot,\cdot]_E, g)$, a linear $E$-connection is admissible if and only if it is $E$-metric-$g$-compatible.
\label{pc4}
\end{Proposition}

\begin{proof} First, we need to observe that every almost-Courant algebroid can be considered as a local almost-Leibniz algebroid with the locality operator
\begin{equation} L(\Omega, u, v) = g(u, v) g^{-1}(\Omega), \nonumber
\end{equation}
so that the admissibility condition (\ref{ec8}) becomes
\begin{align} 
g(\nabla_{X_a} u, v) g^{-1}(e^a) + g(\nabla_{X_a} v, u) g^{-1}(e^a) &= g^{-1} (D(g(u, v))) \nonumber\\
&= g^{-1}(\rho(X_a)(g)(u, v) e^a) \nonumber\\
&= \rho(X_a)(g(u, v)) g^{-1}(e^a). \nonumber
\end{align}
This implies that (and is implied by)
\begin{equation} \rho(X_a)(g(u, v)) - g(\nabla_{X_a} u, v) - g(u, \nabla_{X_a} v) = 0, \nonumber
\end{equation}
which becomes by using the $C^{\infty}(M, \mathbb{R})$-linearity of the anchor $\rho$, the $E$-metric $g$ and the first component of the linear $E$-connection $\nabla$
\begin{equation} Q(\nabla, g)(w, u, v) = 0, \nonumber
\end{equation}
so that a linear $E$-connection $\nabla$ is admissible if and only if it is $E$-metric-$g$-compatible.
\end{proof}

\noindent Courant algebroid connections are usually defined to be $E$-metric-$g$-compatible in the literature, \cite{22}. This compatibility condition is  noted as required for the nice properties such as the anti-symmetry of the $E$-torsion. Hence, the admissibility of $\nabla$ gives a more comprehensive explanation about why this is the case.

\item \textbf{Higher-Courant algebroids:} Motivated from the standard exact Courant algebroid of the form $\mathbb{T}(M) := T(M) \oplus T^*(M)$, one can define the higher analogues of Courant algebroids. 
\begin{Definition} The quartet $(\mathbb{T}^p(M), \rho, [.,.]_D, g)$ is called a (standard) higher-Courant algebroid where $\mathbb{T}^p(M) := T(M) \oplus \Lambda^p(T^*(M))$, the anchor $\rho$ is given by the projection onto the first component, higher-Dorfman bracket $[.,.]_D$ and the $\Lambda^{p-1}(T^*(M))$-valued $E$-metric $g$ are defined by
\begin{align} 
&[U + \omega, V + \eta]_D := [U, V] + \mathcal{L}_U \eta - \iota_V d \omega, \nonumber\\
&g(U + \omega, V + \eta) := \iota_U \eta + \iota_V \omega,
\label{ec30}
\end{align}
for all $U, V \in \mathfrak{X}(M), \omega, \eta \in \Omega^p(M)$ \cite{8}, \cite{13}.
\label{dc12}
\end{Definition}
\noindent As we will show in the proof of the following proposition, $(\mathbb{T}^p(M), \rho, [.,.]_D)$ is a local pre-Leibniz algebroid.

\begin{Proposition} On a higher-Courant algebroid $(\mathbb{T}^p(M), \rho, [.,.]_D, g)$, a linear $E$-connection is admissible if and only if $\nabla$ is $\Lambda^{p-1}(T^*(M))$-valued $E$-metric-$g$-compatible in the sense that
\begin{equation} \iota_w d(g(u, v)) = g(\nabla_w u, v) + g(u, \nabla_w v),
\label{ec31}
\end{equation}
for all $u, v, w \in \mathfrak{X}(E)$.
\label{pc5}
\end{Proposition}

\begin{proof} We start by observing that any higher-Courant algebroid is a local pre-Leibniz algebroid with the locality operator
\begin{equation} ^g L(\Omega, u, v) := pr_1^*(\Omega) \wedge g(u, v), \nonumber
\end{equation}
where $pr_1^*: \mathbb{T}^{p*}(M) := T^*(M) \oplus \Lambda^p(T(M)) \to T^*(M)$ is the projection onto the first component, which follows from the definition of the higher-Dorfman bracket (\ref{ec30}). Next, by using the Cartan magic formula (\ref{eb2}) for usual $p$-forms we can show that
\begin{equation} [u, v]_D + [v, u]_D = [U + \omega, V + \eta]_D + [V + \eta, U + \omega]_D = d(g(u, v)), \nonumber
\end{equation}
where the anti-symmetry of the Lie bracket is also used. Then, the admissibility condition (\ref{ec8}) for a linear $E$-connection $\nabla$ reads
\begin{align}
d(g(u, v)) &= L(e^a, \nabla_{X_a} u, v) + L(e^a, \nabla_{X_a} v, u) \nonumber\\
&= pr_1^*(e^a) \wedge g(\nabla_{X_a} u, v) + pr_1^*(e^a) \wedge g(u, \nabla_{X_a} v). \nonumber
\end{align}
Now we observe that a local frame of $\mathbb{T}^p(M)$ is of the form $X_a = (x_r, E^{r_1} \wedge \ldots \wedge E^{r_p})$ where $(x_r)$ is a usual local frame on the tangent bundle $T(M)$ whose dual local coframe is $(E^r)$. Similarly, a local frame of $\mathbb{T}^{p*}(M)$ is of the form $e^a = (E^r, x^{r_1 \ldots r_p})$ where $(x^{r_1 \ldots r_p})$ is a local frame for the multi-vectors $\Lambda^p(T(M))$. For the anchor and projection maps we have:
\begin{align}
& \rho(X_a) = x_r, \nonumber\\
& pr_1^*(e^a) = E^r. \nonumber
\end{align}
Combining these with the following fact for a usual $(p-1)$-form $\omega \in \Omega^{p-1}(M)$
\begin{equation} d \omega = E^r \wedge \iota_{x_r} d \omega, \nonumber\\
\end{equation}
we get
\begin{equation} \iota_{X_a} d(g(u, v)) = g(\nabla_{X_a} u, v) + g(u, \nabla_{X_a} v). \nonumber
\end{equation}
By the $C^{\infty}(M, \mathbb{R})$-linearity of the $E$-metric $g$, interior product and the first component of a linear $E$-connection, we get the desired result.

\end{proof}
Note that for a consistency check for the Courant algebroid case when $p = 1$, we have:
\begin{equation} \iota_{\rho(w)} d(g(u, v)) = d(g(u, v))(\rho(w)) = \rho(w)(g(u, v)),
\label{ec31b}
\end{equation}
which is the required term for Courant algebroids. Moreover, the chosen locality operator $^g L$ in the above proof satisfies $\rho L(\Omega, u, v)$ for all $\Omega \in \Omega^1(E), u, v \in \mathfrak{X}(E)$, so that one does not need a locality projector. 

\item \textbf{Conformal Courant algebroids:} One can generalize the notion of Courant algebroids, which are defined with respect to an $E$-metric, to conformal Courant algebroids in which the defining notion is a conformal structure depending on a line bundle.
\begin{Definition} A sextet $(E, \rho, [.,.]_E, R, g, ^R \nabla)$ is called a conformal Courant algebroid if $(E, \rho, [.,.]_E)$ is a Leibniz algebroid over $M$ (i.e. an almost-Leibniz algebroid whose bracket satisfies the Leibniz identity), $R$ is a line bundle over $M$, $g$ is an $R$-valued $E$-metric and $^R \nabla$ is an $E$-connection on $R$ satisfying
\begin{align}
&^R \nabla_u (g(v, w)) = g([u, v]_E, w) + g(v, [u, w]_E), \nonumber\\
&[u, v]_E + [v, u]_E = g^{-1} \left( ^R \nabla(g(u, v)) \right),
\label{ec32}
\end{align}
for all $u, v, w \in \mathfrak{X}(E)$ \cite{10}.
\label{dc13}
\end{Definition}

\begin{Proposition} On a conformal Courant algebroid $(E, \rho, [.,.]_E, R, g, ^R \nabla)$, a linear $E$-connection $\nabla$ is admissible if and only if $\nabla$ is $R$-valued $E$-metric-$g$-compatible in the sense that
\begin{equation} ^R \nabla_u (g(v, w)) = g(\nabla_u v, w) + g(v, \nabla_u w),
\label{ec33}
\end{equation}
for all $u, v, w \in \mathfrak{X}(E)$.
\label{pc6}
\end{Proposition}

\begin{proof} Similar to the higher-Courant algebroids, we start with observing that any conformal Courant algebroid is a local pre-Leibniz algebroid with the locality operator
\begin{equation} ^R L(\Omega, u, v) := g^{-1}(g(u, v) \otimes \Omega).\nonumber
\end{equation}
By the defining property (\ref{ec32}), the admissibility condition (\ref{ec8}) for a linear $E$-connection $\nabla$ reads
\begin{align} 
g^{-1} \left( ^R \nabla(g(v, w)) \right) &= L(e^a, \nabla_{X_a} v, w) + L(e^a, \nabla_{X_a} w, v) \nonumber\\
&= g^{-1}(g(\nabla_{X_a} v, w) \otimes e^a) + g^{-1}(g(\nabla_{X_a} w, v) \otimes e^a) \nonumber\\
&= g^{-1} \left( \left[ g(\nabla_{X_a} v, w) + g(\nabla_{X_a} w, v) \right] \otimes e^a \right), \nonumber
\end{align}
which implies 
\begin{equation} ^R \nabla_{X_a} (g(v, w)) = g(\nabla_{X_a} v, w) + g(v, \nabla_{X_a} w). \nonumber
\end{equation}
Then, by using the $C^{\infty}(M, \mathbb{R})$-linearity of $\rho, g$ and the first component of an $E$-connection, we get the desired result.
\end{proof}
\end{itemize}

After proving that many algebroids in the literature yield admissible linear $E$-connections in the way that they are only metric-compatible in some generalized sense, from now on, we will try to promote the usefulness of these admissible linear $E$-connections. First useful identities are the generalizations of the Bianchi identities:
\begin{Proposition} If $\nabla$ is an admissible linear $E$-connection, then its $E$-curvature and $E$-torsion operators satisfy the following first and second algebraic Bianchi identities
\begin{align} 
&R(\nabla)(u, v)w + cycl. = (\nabla_u T(\nabla))(v, w) + T(\nabla)(T(\nabla)(u, v), w) \nonumber\\
& \qquad \qquad \qquad \qquad \quad \ + \nabla_{(1 - \mathcal{P}) L(e^a, \nabla_{X_a} u, v)} w + [[u, v]_E^{\nabla}, w]_E^{\nabla} + cycl., \nonumber\\
&(\nabla_u R(\nabla))(v, w) w' + cycl. = R(\nabla)(u, T(\nabla)(v, w)) w' + \nabla_u \nabla_{(1 - \mathcal{P}) L(e^a, \nabla_{X_a} u, v)} w' \nonumber\\
& \qquad \qquad \qquad \qquad \qquad \qquad - \nabla_{(1 - \mathcal{P}) L(e^a, \nabla_{X_a} u, v)} \nabla_u w' - \nabla_{(1 - \mathcal{P}) L(e^a, \nabla_{X_a} [v, w]_E^{\nabla}, u)} w' \nonumber\\
& \qquad \qquad \qquad \qquad \qquad \qquad + \nabla_{[[u, v]_E^{\nabla}, w]_E^{\nabla}}w' + cycl.,
\label{ec34}
\end{align}
where $cycl.$ means the addition of cyclic permutations in $u, v, w$.
\label{pc7}
\end{Proposition}

\begin{proof} We will skip this proof, but we will prove similar Bianchi identities in the proposition (\ref{pc10}) whose proof can be applied to this proposition. We note that the admissibility part is crucial as it implies the modified and projected modified brackets are anti-symmetric. 
\end{proof}

\noindent Note that these identities reduce to the usual Bianchi identities (\ref{eb7}) for the tangent bundle, which is a Lie algebroid. As the tangent bundle is an almost-Lie algebroid, it is an almost-dull algebroid so that the locality operator is 0. Hence the terms which include $L$ in the equation (\ref{ec34}) vanish. Moreover, as the Lie bracket satisfies the Jacobi identity, the term $[[u, v], w] + cycl.$ is automatically 0. Hence, the only remaining terms are equal to the ones in the equations (\ref{eb7}). Moreover, both in double field theory \cite{5} and Courant algebroid \cite{22} literature, there are various versions of the algebraic Bianchi identities. These are often written down by using the isomorphism induced by the $E$-metric, and they include a different notion of $E$-curvature operator. Still, the necessary ingredient for these identities is the fact that the chosen connection is compatible with the $E$-metric which is either the $O(d, d)$-metric in double field theory or the one in the definition of a Courant algebroid. The admissibility condition of $\nabla$ explains why this $E$-metric-compatibility is necessary for the Bianchi identities.

\begin{Proposition} For any linear $E$-connection $\nabla$, not necessarily admissible, the following Ricci identity holds:
\begin{equation} \nabla^2_{u, v} w - \nabla^2_{v, u} w = R(\nabla)(u, v)w - \nabla_{T(\nabla)(u, v)} w + \nabla_{(1 - \mathcal{P}) L(e^a, \nabla_{X_a} u, v)} w,
\label{ec35}
\end{equation}
where the second order $E$-covariant derivative is defined in the same way as the equation (\ref{eb9}).
\label{pc8}
\end{Proposition}

\begin{proof} We will prove this by direct computation and using the equation (\ref{ec23}):
\begin{align} 
\nabla^2_{u, v} w - \nabla^2_{v, u} w &= \left\{ \nabla_u \nabla_v w - \nabla_{\nabla_u v} w \right\} - \left\{ \nabla_v \nabla_u w - \nabla_{\nabla_v u} w \right\} \nonumber\\
&= \nabla_u \nabla_v w - \nabla_v \nabla_u w - \nabla_{\nabla_u v - \nabla_v u} w \nonumber\\
&= R(\nabla)(u, v) w + \nabla_{[u, v]_E} w + \nabla_{(1 - \mathcal{P})L(e^a, \nabla_{X_a} u, v)} w - \nabla_{T(\nabla)(u, v) + [u, v]_E} w \nonumber\\
&= R(\nabla)(u, v) w - \nabla_{T(\nabla)(u, v)} w + \nabla_{(1 - \mathcal{P})L(e^a, \nabla_{X_a} u, v)} w \nonumber
\end{align}
\end{proof}

\noindent Similar to the Bianchi identities, this generalized Ricci identity reduces to the usual one (\ref{eb8}) when the locality operator vanishes, which is the case for the tangent bundle. Moreover, when one deals with the $E$-torsion-free linear $E$-connections, one can consider 
\begin{equation}
R(\nabla)(u, v)w = \nabla^2_{u, v} w - \nabla^2_{v, u} w - \nabla_{(1 - \mathcal{P}) L(e^a, \nabla_{X_a} u, v)} w,
\label{ec36}
\end{equation}
as the definition of the $E$-curvature operator.

\noindent The anti-symmetry properties of $E$-torsion (\ref{ec20}) and $E$-curvature (\ref{ec28}) operators lead us to the following definitions, which are completely analogous to the usual case (\ref{eb16}).

\begin{Definition} For an admissible linear $E$-connection $\nabla$, its $E$-torsion and $E$-curvature 2-forms are defined by
\begin{align}
&T^a(\nabla)(u, v) := \langle e^a, T(\nabla)(u, v) \rangle, \nonumber\\
&R^a_{\ b}(\nabla)(u, v) := \langle e^a, R(\nabla)(u, v) X_b \rangle.
\label{ec37}
\end{align}
Similarly, $E$-non-metricity and $E$-connection 1-forms, which can be constructed for any linear $E$-connection, are defined by
\begin{align}
&Q_{a b}(\nabla, g)(v) := Q(\nabla, g)(v, X_a, X_b), \nonumber\\
&\omega^a_{\ b}(\nabla)(v) := \langle e^a, \nabla_v X_b \rangle.
\label{ec38}
\end{align}
\label{dc14}
\end{Definition}
These definitions naturally raise the question whether the definitions of $E$-torsion and $E$-curvature operators can be written in the Cartan form by using these $E$-torsion and $E$-curvature 2-forms, which are analogous to the equations (\ref{eb13}). The Cartan structure equations necessitates the exterior derivative that can be defined via the relation (\ref{eb1}), which includes the Lie bracket. If one changes the Lie bracket to $[\cdot,\cdot]_E$, then the definition does not give a map between $E$-tensors as it is not $C^{\infty}(M, \mathbb{E})$-multilinear. Similar to $E$-torsion and $E$-curvature operators, one can seek for a suitable tensorial modification:

\begin{align} 
d(\nabla) \Omega (v_1, \ldots, v_{p+1}) := & \sum_{1 \leq i \leq p+1} (-1)^{i+1} \rho(v_i) \left( \Omega(v_1, \ldots, \check{v}_i, \ldots, v_{p+1}) \right) \nonumber\\
& + \sum_{1 \leq i < j \leq p+1} (-1)^{i+j} \Omega \left( [v_i, v_j]_E^{\nabla}, v_1, \ldots, \check{v}_i, \ldots, \check{v}_j, \ldots, v_{p+1} \right).
\label{ec39}
\end{align}
For a general linear $E$-connection, this map does not yield an $E$-form as the right-hand side is not anti-symmetric. Yet, if the connection $\nabla$ is admissible, then it defines a map between $E$-forms. Unfortunately, this map does not square to 0 as for the usual exterior derivative. On the other hand, if one uses the projected modified bracket (\ref{ec24}), then one can still define a tensorial modification whose square is 0 when acting on smooth functions:

\begin{align} 
\hat{d}(\nabla) \Omega (v_1, \ldots, v_{p+1}) := & \sum_{1 \leq i \leq p+1} (-1)^{i+1} \rho(v_i) \left( \Omega(v_1, \ldots, \check{v}_i, \ldots, v_{p+1}) \right) \nonumber\\
& + \sum_{1 \leq i < j \leq p+1} (-1)^{i+j} \Omega \left( [v_i, v_j]_E^{\hat{\nabla}}, v_1, \ldots, \check{v}_i, \ldots, \check{v}_j, \ldots, v_{p+1} \right).
\label{ec40}
\end{align}
\noindent As $(E, \rho, [\cdot,\cdot]_E^{\hat{\nabla}})$ is a pre-Lie algebroid,
\begin{equation} \hat{d}(\nabla)^2 f = 0,
\label{ec41}
\end{equation}
for all $f \in C^{\infty}(M, \mathbb{R})$. As the bracket $[\cdot,\cdot]_E^{\hat{\nabla}}$ fails to satisfy the Jacobi identity,
\begin{equation} \hat{d}(\nabla)^2 \Omega = \Omega \left( Assoc \left( [\cdot,\cdot]_E^{\hat{\nabla}} \right) \right),
\label{ec42}
\end{equation}
where the associator of any bracket $[\cdot,\cdot]'$ on $E$ is defined by
\begin{equation} Assoc \left( [\cdot,\cdot]' \right) (u, v, w) :=  [u, [v, w]']' - [[u, v]', w]' - [v, [u, w]']',
\label{ec43}
\end{equation}
for all $u, v, w \in \mathfrak{X}(E)$, which is not equal to 0 in general \cite{23}.

An almost-Leibniz algebroid with a bracket whose associator is 0, i.e. it satisfies the Leibniz identity is called a Leibniz algebroid. Any Leibniz algebroid is automatically a pre-Leibniz algebroid, as the Leibniz identity implies the condition (\ref{ec21}). For example, Courant algebroids, higher-Courant algebroids are, by definition, Leibniz algebroids. On the other hand, a Leibniz algebroid with an anti-symmetric bracket is a Lie algebroid. Almost every construction on the tangent bundle can be easily elevated to an arbitrary Lie algebroid level \cite{1}, \cite{24}. Here, as the associator of the projected modified bracket $[\cdot,\cdot]_E^{\hat{\nabla}}$ is not 0, one can not have a genuine $E$-exterior derivative whose square is always 0, so that one cannot define a cohomology on $E$-forms. Yet, characteristic classes on pre-Lie algebroids are studied \cite{23}. Moreover, the ``naive cohomology'' on Courant algebroids are investigated by using the Courant bracket instead of the Dorfman bracket and restricting to $E$-forms on the subbundle $ker(\rho)$ \cite{25}, \cite{26}.

\begin{Definition} Let $\nabla$ be an admissible linear $E$-connection. Then, the map $d(\nabla): \Omega^p(E) \to \Omega^{p+1}(E)$ defined in the equation (\ref{ec39}) will be called the $E$-exterior derivative, and the map $\hat{d}(\nabla): \Omega^p(E) \to \Omega^{p+1}(E)$ defined in the equation (\ref{ec40}) will be called the projected $E$-exterior derivative.
\label{dc15}
\end{Definition}
\noindent Note that their actions on the smooth functions agree

\begin{equation} d(\nabla)f (v) = \hat{d}(\nabla)f (v) = Df (v) = \rho(v)(f),
\label{ec44}
\end{equation}
for all $f \in C^{\infty}(M, \mathbb{R}), v \in \mathfrak{X}(E)$. Morover, for pre-Lie algebroids, and in particular for the tangent bundle, both of these maps reduce to the usual exterior derivative $d$ as $L = 0$. Both $E$-exterior derivative and projected $E$-exterior derivative are degree-1 graded derivations, i. e. they satisfy
\begin{align} 
&d(\nabla)(\Omega_1 \wedge \Omega_2) = d(\nabla) \Omega_1 \wedge \Omega_2 + (-1)^{p_1} \Omega_1 \wedge d(\nabla) \Omega_2, \nonumber\\
&\hat{d}(\nabla)(\Omega_1 \wedge \Omega_2) = \hat{d}(\nabla) \Omega_1 \wedge \Omega_2 + (-1)^{p_1} \Omega_1 \wedge \hat{d}(\nabla) \Omega_2
\label{ec45} 
\end{align}
for all $\Omega_i \in \Omega^{p_i}(E)$. Moreover, by using the maps $d(\nabla)$ and $\hat{d}(\nabla)$, one can write down the Cartan structure equations:

\begin{Proposition} For an admissible linear $E$-connection $\nabla$, the following forms of Cartan first and second structure equations hold:
\begin{align} 
&T^a(\nabla) = d(\nabla) e^a + \omega^a_{\ b}(\nabla) \wedge e^b, \nonumber\\
&R^a_{\ b}(\nabla) = \hat{d}(\nabla) \omega^a_{\ b}(\nabla) + \omega^a_{\ c}(\nabla) \wedge \omega^c_{\ b}(\nabla).
\label{ec46}
\end{align}
\label{pc9}
\end{Proposition}

\begin{proof} We will prove the first Cartan structure equation by observing
\begin{align} 
e^a(\nabla_u v) &= e^a(\nabla_u (e^b(v) X_b)) \nonumber\\
&= e^a \left( \rho(u)(e^b(v)) X_b + e^b(v) \nabla_u X_b \right) \nonumber\\
&= \rho(u)(e^b(v)) e^a(X_b) + e^b(v) e^a(\nabla_u X_b) \nonumber\\
&= \rho(u)(e^a(v)) + e^a(\nabla_u X_b) e^b(v). \nonumber
\end{align}
The right-hand side of the first Cartan structure equation reads
\begin{align} 
d(\nabla) e^a(u, v) + (\omega(\nabla)^a_{\ b} \wedge e^b)(u, v) &= \rho(u)(e^a(v)) - \rho(v)(e^a(u)) - e^a \left( [u, v]_E^{\nabla} \right)  \nonumber\\
& \quad \ \omega(\nabla)^a_{\ b}(u) e^b(v) - \omega(\nabla)^a_{\ b}(v) e^b(u) \nonumber\\
&= \rho(u)(e^a(v)) - \rho(v)(e^a(u)) - e^a \left( [u, v]_E^{\nabla} \right)  \nonumber\\
& \quad \ e^a(\nabla_u X_b) e^b(v) - e^a(\nabla_v X_b) e^b(u). \nonumber
\end{align}
With the above observation, this becomes
\begin{align} 
d(\nabla) e^a(u, v) + (\omega(\nabla)^a_{\ b} \wedge e^b)(u, v) &= e^a(\nabla_u v) - e^a(\nabla_v u) - e^a \left( [u, v]_E^{\nabla} \right)  \nonumber\\
&= e^a(T(u, v)) = T^a(\nabla)(u, v). \nonumber
\end{align}
Hence, we proved the first Cartan structure equation. For the second one, we replace $v$ by $\nabla_v X_b$ in the above observation so that the right-hand side can be written as
\begin{align} 
\hat{d}(\nabla) \omega^a_{\ b} (u, v) + (\omega(\nabla)^a_{\ b} \wedge \omega(\nabla)^c_{\ b})(u, v) &= \rho(u)(e^a(\nabla_v X_b)) - \rho(v)(e^a(\nabla_u X_b)) \nonumber\\
& \quad \ - e^a \left( \nabla_{[u, v]_E^{\hat{\nabla}}} X_b \right) +  e^a(\nabla_u X_c) e^c(\nabla_v X_b) \nonumber\\
& \quad \ - e^a(\nabla_v X_c) e^c(\nabla_u X_b) \nonumber\\
&= e^a(\nabla_u \nabla_v X_b) - e^a(\nabla_v \nabla_u X_b) - e^a \left( \nabla_{[u, v]_E^{\hat{\nabla}}} X_b \right) \nonumber\\
&= e^a(R(\nabla)(u, v) X_b) = R^a_{\ b}(\nabla)(u, v). \nonumber
\end{align}
Hence, we also proved the second Cartan structure equation.
\end{proof}

\noindent It seems odd to write down the Cartan structure equations by using two different notions of $E$-exterior derivatives. This is due to the fact that the $E$-torsion and $E$-curvature operators are defined by using different modifications of the bracket. Yet, the only key point in these definitions were to construct tensorial quantities, which is the main aim of the constructions of these operators also in the literature. Many authors defined seemingly unnatural, ad hoc operators in order to achieve tensoriality. By keeping in mind the two-fold Cartan structure equations, we offer another $E$-torsion operator:

\begin{Definition} The projected $E$-torsion operator of a linear $E$-connection $\nabla$ is defined as
\begin{equation} \hat{T}(\nabla)(u, v) := \nabla_u v - \nabla_v u - [u, v]_E^{\hat{\nabla}},
\label{ec47}
\end{equation}
for all $u, v \in \mathfrak{X}(E)$.
\label{dc16}
\end{Definition}
As the projected modified bracket $[\cdot,\cdot]_E^{\hat{\nabla}}$ is an almost-dull bracket, the projected $E$-torsion operator is $C^{\infty}(M, \mathbb{R})$-bilinear. If $\nabla$ is admissible, the projected modified bracket is an almost-Lie bracket so that the projected $E$-torsion operator is anti-symmetric. Therefore, in this case one can define the projected $E$-torsion 2-forms:

\begin{equation} \hat{T}^a(\nabla)(u, v) := \langle e^a, \hat{T}(\nabla)(u, v) \rangle.
\label{ec48}
\end{equation}
\noindent In terms of these projected $E$-torsion 2-forms, the first Cartan structure equation becomes

\begin{equation} \hat{T}(\nabla) = \hat{d}(\nabla) e^a + \omega^a_{\ b}(\nabla) \wedge e^b.
\label{ec49}
\end{equation}
\noindent Moreover, the algebraic Bianchi identities (\ref{ec34}) take a more familiar form as the usual ones (\ref{eb7}):

\begin{Proposition} For an admissible linear $E$-connection $\nabla$, the following algebraic Bianchi identities hold:
\begin{align} 
&R(\nabla)(u, v) w + cycl. = (\nabla_u \hat{T}(\nabla))(v, w) + \hat{T}(\nabla)(\hat{T}(\nabla)(u, v), w)
+ [[u, v]_E^{\hat{\nabla}}, w]_E^{\hat{\nabla}} + cycl., \nonumber\\
&(\nabla_u R(\nabla))(v, w) w' + cycl. = R(\nabla)(u, \hat{T}(\nabla)(v, w)) w' + \nabla_{[[u, v]_E^{\hat{\nabla}}, w]_E^{\hat{\nabla}}} w' + cycl.,
\label{ec50}
\end{align}
for all $u, v, w, w' \in \mathfrak{X}(E)$.
\label{pc10}
\end{Proposition}

\begin{proof} We start with the definition of the $E$-curvature operator of the form (\ref{ec25})
\begin{align} 
R(\nabla)(u, v) w + cycl. &= \nabla_u \nabla_v w - \nabla_v \nabla_u w - \nabla_{[u, v]_E^{\hat{\nabla}}} w \nonumber\\
& \quad + \nabla_v \nabla_w u - \nabla_w \nabla_v u - \nabla_{[v, w]_E^{\hat{\nabla}}} u \nonumber\\
& \quad + \nabla_w \nabla_u v - \nabla_u \nabla_w v - \nabla_{[w, u]_E^{\hat{\nabla}}} v \nonumber\\
&= \nabla_u \left( \nabla_v w - \nabla_w v \right) + \nabla_v \left( \nabla_w u - \nabla_u w \right) + \nabla_w \left( \nabla_u v - \nabla_v u \right) \nonumber\\
& \quad - \nabla_{[u, v]_E^{\hat{\nabla}}} w - \nabla_{[v, w]_E^{\hat{\nabla}}} u - \nabla_{[w, u]_E^{\hat{\nabla}}} v \nonumber\\
&= \nabla_u \left( \hat{T}(\nabla)(v, w) + [v, w]_E^{\hat{\nabla}} \right) + \nabla_v \left( \hat{T}(\nabla)(w, u) + [w, u]_E^{\hat{\nabla}} \right) \nonumber\\
& \quad + \nabla_w \left( \hat{T}(\nabla)(u, v) + [u, v]_E^{\hat{\nabla}} \right) - \nabla_{[u, v]_E^{\hat{\nabla}}} w - \nabla_{[v, w]_E^{\hat{\nabla}}} u - \nabla_{[w, u]_E^{\hat{\nabla}}} v \nonumber
\end{align}
Now we observe that 
\begin{align} 
&\nabla_w [u, v]_E^{\hat{\nabla}} - \nabla_{[u, v]_E^{\hat{\nabla}}} w = \hat{T}(\nabla)(w, [u, v]_E^{\hat{\nabla}}) + [w, [u, v]_E^{\hat{\nabla}}]_E^{\hat{\nabla}} \nonumber\\
&(\nabla_u \hat{T}(\nabla))(v, w) = \nabla_u (\hat{T}(\nabla)(v, w)) - \hat{T}(\nabla)(\nabla_u v, w) - \hat{T}(\nabla)(v, \nabla_u w), \nonumber
\end{align}
which implies 
\begin{align}
R(\nabla)(u, v) w + cycl. &= (\nabla_u \hat{T}(\nabla))(v, w) + \hat{T}(\nabla)(\nabla_u v, w) + \hat{T}(\nabla)(v, \nabla_u w) \nonumber\\
& \quad + (\nabla_v \hat{T}(\nabla))(w, u) + \hat{T}(\nabla)(\nabla_v w, u) + \hat{T}(\nabla)(w, \nabla_v u) \nonumber\\
& \quad + (\nabla_w \hat{T}(\nabla))(u, v) + \hat{T}(\nabla)(\nabla_w u, v) + \hat{T}(\nabla)(u, \nabla_w v) \nonumber\\
& \quad + \hat{T}(\nabla)(u, [v, w]_E^{\hat{\nabla}}) + \hat{T}(\nabla)(v, [w, u]_E^{\hat{\nabla}}) + \hat{T}(\nabla)(w, [u, v]_E^{\hat{\nabla}}) \nonumber\\
& \quad + [u, [v, w]_E^{\hat{\nabla}}]_E^{\hat{\nabla}} + [v, [w, u]_E^{\hat{\nabla}}]_E^{\hat{\nabla}} + [w, [u, v]_E^{\hat{\nabla}}]_E^{\hat{\nabla}}. \nonumber
\end{align}
By using the fact that the projected $E$-torsion operator is anti-symmetric for an admissible linear $E$-connection, we get
\begin{equation}
R(\nabla)(u, v) w + cycl. = (\nabla_u \hat{T}(\nabla))(v, w) + \hat{T}(\nabla)(\nabla_u v - \nabla_v u - [u, v]_E^{\hat{\nabla}}, w) + [u, [v, w]_E^{\hat{\nabla}}]_E^{\hat{\nabla}} + cycl., \nonumber
\end{equation}
which is the desired first Bianchi identity. For the second identity, we start with

\begin{align}
(\nabla_u R(\nabla))(v, w) w' &= \nabla_u (R(\nabla)(v, w) w') - R(\nabla)(\nabla_u v, w) w' - R(\nabla)(v, \nabla_u w) w' - R(\nabla)(v, w) \nabla_u w' \nonumber\\
&= \nabla_u \left( \nabla_v \nabla_w w' - \nabla_w \nabla_v w' - \nabla_{[v, w]_E^{\hat{\nabla}}} w' \right) \nonumber\\
& \quad - \ R(\nabla)(\nabla_u v, w) w' - R(\nabla)(v, \nabla_u w) w' - R(\nabla)(v, w) \nabla_u w' \nonumber\\
&= \nabla_u \nabla_v \nabla_w w' - \nabla_u \nabla_w \nabla_v w' - \nabla_u \nabla_{[v, w]_E^{\hat{\nabla}}} w' \nonumber\\
& \quad \ - R(\nabla)(\nabla_u v, w) w' - R(\nabla)(v, \nabla_u w) w' - R(\nabla)(v, w) \nabla_u w' \nonumber
\end{align}
Next, we observe
\begin{align} 
&R(\nabla)([v, w]_E^{\hat{\nabla}}, u) w' = \nabla_{[v, w]_E^{\hat{\nabla}}} \nabla_u w' - \nabla_u \nabla_{v, w]_E^{\hat{\nabla}}} w' - \nabla_{[[v, w]_E^{\hat{\nabla}}, u]_E^{\hat{\nabla}}} w', \nonumber\\
&R(\nabla)(v, w) \nabla_u w' = \nabla_v \nabla_w \nabla_u w' - \nabla_w \nabla_v \nabla_u - \nabla_{[v, w]_E^{\hat{\nabla}}} \nabla_u w', \nonumber
\end{align}
which implies 
\begin{align}
&\nabla_u \nabla_{[v, w]_E^{\hat{\nabla}}} w' = - R(\nabla)([v, w]_E^{\hat{\nabla}}, u) w' + \nabla_{[v, w]_E^{\hat{\nabla}}} \nabla_u w' - \nabla_{[[v, w]_E^{\hat{\nabla}}, u]_E^{\hat{\nabla}}} w', \nonumber\\
&\nabla_v \nabla_w \nabla_u w' - \nabla_w \nabla_v \nabla_u w' = R(\nabla)(v, w) \nabla_u w' + \nabla_{[v, w]_E^{\hat{\nabla}}} \nabla_u w'. \nonumber
\end{align}
Inserting these in the cyclic combination and canceling out, we get
\begin{align} 
(\nabla_u R(\nabla))(v, w) w' + cycl. &= R(\nabla)([v, w]_E^{\hat{\nabla}}, u) w' + \nabla_{[[v, w]_E^{\hat{\nabla}}, u]_E^{\hat{\nabla}}} \nonumber\\
& \quad \ + R(\nabla)([w, u]_E^{\hat{\nabla}}, v) w' + \nabla_{[[w, u]_E^{\hat{\nabla}}, v]_E^{\hat{\nabla}}} w' \nonumber\\
& \quad \ + R(\nabla)([u, v]_E^{\hat{\nabla}}, w) w' + \nabla_{[[u, v]_E^{\hat{\nabla}}, w]_E^{\hat{\nabla}}} w' \nonumber\\
& \quad \ - R(\nabla)(\nabla_u v, w) w' - R(\nabla)(w, \nabla_v u) w' \nonumber\\
& \quad \ - R(\nabla)(\nabla_v w, u) w' - R(\nabla)(u, \nabla_w v) w' \nonumber\\
& \quad \ - R(\nabla)(\nabla_w u, v) w' - R(\nabla)(v, \nabla_u w) w'. \nonumber
\end{align}
Now, as the $E$-curvature operator is anti-symmetric for an admissible linear $E$-connection, we get
\begin{align} 
(\nabla_u R(\nabla))(v, w) w' + cycl. &=  R(\nabla)(u, \nabla_v w) w' - R(\nabla)(u, \nabla_w v) w' - R(\nabla)(u, [v, w]_E^{\hat{\nabla}}) w' \nonumber\\
& \quad \ + R(\nabla)(v, \nabla_w u) w' - R(\nabla)(v, \nabla_u w) w' - R(\nabla)(v, [w, u]_E^{\hat{\nabla}}) w' \nonumber\\
& \quad \ + R(\nabla)(w, \nabla_u v) w' - R(\nabla)(w, \nabla_v u) w' - R(\nabla)(w, [u, v]_E^{\hat{\nabla}}) w' \nonumber\\
& \quad \ + \nabla_{[[u, v]_E^{\hat{\nabla}}, w]_E^{\hat{\nabla}}} w' + \nabla_{[[v, w]_E^{\hat{\nabla}}, u]_E^{\hat{\nabla}}} + \nabla_{[[w, u]_E^{\hat{\nabla}}, v]_E^{\hat{\nabla}}} w' \nonumber\\
&= R(\nabla)(u, \hat{T}(\nabla)(v, w)) w' + \nabla_{[[u, v]_E^{\hat{\nabla}}, w]_E^{\hat{\nabla}}} w' + cycl., \nonumber	
\end{align}
which is the desired second Bianchi identity.
\end{proof}

\noindent These are identical to the usual Bianchi identities except for the associator term. Similarly, the Ricci identity (\ref{ec35}) becomes
\begin{equation} \nabla^2_{u, v} w - \nabla^2_{v, u} w = R(\nabla)(u, v) w - \nabla_{\hat{T}(\nabla)(u, v)} w.
\label{ec51}
\end{equation}
\noindent Hence, in the absence of projected $E$-torsion, the $E$-curvature operator can be defined as the difference of second order covariant derivatives. As this form of the identities corresponds to Bianchi identities of pre-Lie algebroids, the first Bianchi identity together with the second Cartan equation was anticipated in \cite{23}. In terms of the projected $E$-torsion, one can also write down the differential Bianchi identities.

\begin{Proposition} For an admissible linear $E$-connection $\nabla$ the following differential Bianchi identities hold:
\begin{align} 
&\hat{d}(\nabla) \hat{T}^a(\nabla) + \omega^a_{\ b}(\nabla) \wedge \hat{T}^b(\nabla) = R^a_{\ b}(\nabla) \wedge e^b + \hat{d}(\nabla)^2 e^a, \nonumber\\
&\hat{d}(\nabla) R^a_{\ b}(\nabla) + \omega^a_{\ c}(\nabla) \wedge R^c_{\ b}(\nabla) = R^a_{\ c}(\nabla) \wedge \omega^c_{\ b}(\nabla) + \hat{d}(\nabla)^2 \omega^a_{\ b}(\nabla).
\label{ec52} 
\end{align}
\label{pc11}
\end{Proposition}

\begin{proof}  These differential identities can be proven by taking the projected $E$-exterior derivative of Cartan structure equations. Let us start with the first one by taking the projected $E$-exterior derivative of the equation (\ref{ec49}):
\begin{align}
\hat{d}(\nabla) \hat{T}(\nabla) &= \hat{d}(\nabla) \left\{ \hat{d}(\nabla) e^a + \omega^a_{\ b}(\nabla) \wedge e^b \right\} \nonumber\\
&= \hat{d}(\nabla)^2 e^a + \hat{d}(\nabla) \omega^a_{\ b}(\nabla) \wedge e^b - \omega^a_{\ b}(\nabla) \wedge \hat{d}(\nabla) e^b \nonumber\\
&= \hat{d}(\nabla)^2 e^a + \hat{d}(\nabla) \omega^a_{\ b}(\nabla) \wedge e^b - \omega^a_{\ b}(\nabla) \wedge \left\{ \hat{T}^b(\nabla) - \omega^b_{\ c}(\nabla) \wedge e^c \right\} \nonumber\\
&= \hat{d}(\nabla)^2 e^a + \left\{ \hat{d}(\nabla) \omega^a_{\ b}(\nabla) + \omega^a_{\ c}(\nabla) \wedge \omega^c_{\ b}(\nabla) \right\} \wedge e^b - \omega^a_{\ b}(\nabla) \wedge \hat{T}^b(\nabla) \nonumber\\
&= \hat{d}(\nabla)^2 e^a + R^a_{\ b}(\nabla) \wedge e^b - \omega^a_{\ b}(\nabla) \wedge \hat{T}^b(\nabla), \nonumber
\end{align}
which implies the first differential Bianchi identity. For the second equation, we take the projected $E$-exterior derivative of the second Cartan structure equation (\ref{ec46}):
\begin{align}
\hat{d}(\nabla) R^a_{\ b}(\nabla) &= \hat{d}(\nabla) \left\{ \hat{d}(\nabla) \omega^a_{\ b}(\nabla) + \omega^a_{\ c}(\nabla) \wedge \omega^c_{\ b}(\nabla) \right\} \nonumber\\
&= \hat{d}(\nabla)^2 \omega^a_{\ b}(\nabla) + \hat{d}(\nabla) \omega^a_{\ c}(\nabla) \wedge \omega^c_{\ b}(\nabla) - \omega^a_{\ c}(\nabla) \wedge \hat{d}(\nabla) \omega^c_{\ b}(\nabla) \nonumber\\
&= \hat{d}(\nabla)^2 \omega^a_{\ b}(\nabla) + \left\{ R^a_{\ c}(\nabla) - \omega^a_{\ d}(\nabla) \wedge \omega^d_{\ c}(\nabla) \right\} \wedge \omega^c_{\ b}(\nabla) \nonumber\\
& \qquad \qquad \qquad \quad - \omega^a_{\ c}(\nabla) \wedge \left\{ R^c_{\ b}(\nabla) - \omega^c_{\ d}(\nabla) \wedge \omega^d_{\ b}(\nabla) \right\} \nonumber\\
&= \hat{d}(\nabla)^2 \omega^a_{\ b}(\nabla) + R^a_{\ c}(\nabla) \omega^c_{\ b}(\nabla) - \omega^a_{\ c}(\nabla) \wedge R^c_{\ b}(\nabla) \nonumber
\end{align}
which implies the second differential Bianchi identity.
\end{proof}

\noindent Note that both of these equations are identical with (\ref{eb17}), when $\hat{d}(\nabla)^2 = 0$ which is the case when the associator of the projected modified bracket vanishes.

With the projected $E$-torsion operator, the geometric meaning becomes even more transparent. Recall that, for an admissible linear $E$-connection $\nabla$, the projected modified bracket $[\cdot,\cdot]_E^{\hat{\nabla}}$ is a pre-Lie bracket by the proposition (\ref{pc3}). Both the projected $E$-torsion operator $\hat{T}(\nabla)$ (\ref{ec47}) and the $E$-curvature operator $R(\nabla)$ (\ref{ec22}) are just the usual $E$-torsion and $E$-curvature operators on the pre-Lie algebroid $(E, \rho, [\cdot,\cdot]_E^{\hat{\nabla}})$. This pre-Lie algebroid is constructed from the initial local pre-Leibniz algebroid $(E, \rho, [\cdot,\cdot]_E, L)$ with two additional structures:
\begin{enumerate}
\item An admissible linear $E$-connection $\nabla$,
\item A locality projector $\mathcal{P}$.
\end{enumerate}
Hence, the choice of such doublet $(\nabla, \mathcal{P})$ gives the pre-Lie algebroid in which the $E$-torsion, $E$-curvature operators and the $E$-exterior derivative are defined naturally. If the locality projector $\mathcal{P}$ is fixed (for example, for transitive pre-Leibniz algebroids of the form $T(M) \oplus ker(\rho)$ it can be naturally taken as the projection onto the second component), then each admissible linear $E$-connection $\nabla$ yields a pre-Lie algebroid structure. In other words, the equivalence class $[\nabla]_L$ of admissible linear $E$-connections corresponds to an equivalence class of pre-Lie algebroids. Whether this correspondence is one-to-one or not is related to the ``invertibility'' of the locality operator $L$. For example, if one starts with a pre-Lie algebroid, then $L = 0$ so that every linear $E$-connection is admissible, and the corresponding equivalence class of pre-Lie algebroids has only one element which is the original pre-Lie algebroid. The properties of this correspondence will be investigated more comprehensively in future. For example, whether it has some relations to the correspondence between Lie bialgebroids and Courant algebroids \cite{3} looks interesting. Also note that the anomalous terms in both the projected exterior derivative and Bianchi identities are about the associator, so a task about which admissible linear $E$-connections yield a Lie algebroid structure in which the associator vanishes is important.

This correspondence is directly related to the rank of the associated bundle of $E$-Levi-Civita connections, which are defined as $E$-metric-compatible and $E$-torsion-free (see \cite{27} for generalized Levi-Civita connections, and \cite{22} for the calculation of this rank for $E$-Levi-Civita connections on Courant algebroids). This is because each corresponding almost-Lie algebroid has its own unique $E$-Levi-Civita connection. All of these $E$-Levi-Civita connections on individual almost-Lie algebroids constitute a superset of $E$-Levi-Civita connections on the original anti-commutable pre-Leibniz algebroid. One natural way to seek for an $E$-Levi-Civita connection was proposed in our previous paper \cite{12} in which we modified the Koszul formula (\ref{eb18}) in order to find an equation with appropriate properties. If the left-hand side of the usual formula is modified with the following additional term:
\begin{equation} - g(L(e^a, \nabla_{X_a} v, w), u) - g(L(e^a, \nabla_{X_a} u, w), v) + g(L(e^a, \nabla_{X_a} u, v), w),
\label{ec55}
\end{equation}
then $\nabla$ satisfies the necessary properties for a linear $E$-connection. We called such linear $E$-connections that satisfy the modification of the Koszul formula as ``$E$-\textit{Koszul connections}''. Unfortunately, these connections are not $E$-Levi-Civita connections in general. Yet, one can note that these terms (\ref{ec55}) are exactly the necessary ones in order to write down the modified Koszul formula in the following form:
\begin{align} 
2 g(\nabla_u v, w) = & \ \rho(u)(g(v, w)) + \rho(v)(g(u, w)) - \rho(w)(g(u, v)) \nonumber\\
												& - g([v, w]_E^{\nabla}, u) - g([u, w]_E^{\nabla}, v) + g([u, v]_E^{\nabla}, w).
\label{ec56}
\end{align}

\begin{Proposition} A linear $E$-connection is $E$-Levi-Civita if and only if it is admissible and $E$-Koszul.
\label{pc13}
\end{Proposition}

\begin{proof} First, let us assume that $\nabla$ is an $E$-Koszul connection. One can directly evaluate the components of the $E$-torsion and $E$-non-metricity tensors:
\begin{align} 
&T(\nabla)^a_{\ b c} = \frac{1}{2} \left[ \Gamma(\nabla)^e_{\ d c} L^{a d}_{\ \ e b} + \Gamma(\nabla)^e_{\ d b} L^{a d}_{\ \ e c} - \gamma^a_{\ b c} - \gamma^a_{\ c b} \right], \nonumber\\
&Q(\nabla, g)_{a b c} = \frac{1}{2} \left[ \gamma^f_{\ b c} + \gamma^f_{\ c b} - \Gamma(\nabla)^e_{\ d c} L^{f d}_{\ \ e b} - \Gamma(\nabla)^e_{\ d b} L^{f d}_{\ \ e c}  \right] g_{f a}. \nonumber
\end{align}
By using the modified anholonomy coefficients, these can be written as
\begin{align}
&T(\nabla)^a_{\ b c} = \frac{1}{2} \left\{ \gamma(\nabla)^a_{\ b c} - \gamma(\nabla)^a_{\ c b} \right\}, \nonumber\\
&Q(\nabla, g)_{a b c} = \frac{1}{2} \left\{ \gamma(\nabla)^f_{\ b c} - \gamma(\nabla)^f_{\ c b} \right\} g_{f a}. \nonumber
\end{align} 
For an admissible linear $E$-connection the modified bracket is anti symmetric so that the modified anholonmy coefficients are anti-symmetric in their lower indices. Hence, the $E$-torsion and $E$-non-metricity components vanish for an admissible $E$-Koszul connection, so it is an $E$-Levi-Civita connection. 

For the other direction, we use the fact that an $E$-Levi-Civita connection $\nabla$ is $E$-torsion-free by definition, and by the proposition (\ref{pc1b}), it has to be admissible. In order to show it is also $E$-Koszul, we use the ideas identical to the proof of the usual Koszul formula (\ref{eb18}). We start with the following fact due to $E$-metric-$g$-compatibility
\begin{equation} \rho(g(v, w)) = g(\nabla_u v, w) + g(v, \nabla_u w). \nonumber
\end{equation}
By using the $E$-torsion-freeness condition, this becomes
\begin{equation} \rho(u)(g(v, w)) = g(\nabla_v u, w) + g(\nabla_u w, v) + g([u, v]_E^{\nabla}, w). \nonumber
\end{equation}
Calculating the cyclic permutation in $u, v, w$ and evaluating 
\begin{equation} \rho(u)(g(v, w)) + \rho(v)(g(w, u)) - \rho(w)(g(u, v)), \nonumber
\end{equation}
we get
\begin{align}
2 g(\nabla_u v, w) = & \ \rho(u)(g(v, w)) + \rho(v)(g(u, w)) - \rho(w)(g(u, v)) \nonumber\\
												& + g([w, v]_E^{\nabla}, u) - g([u, w]_E^{\nabla}, v) - g([v, u]_E^{\nabla}, w). \nonumber
\end{align}
As we already proved that $\nabla$ is admissible, the anti-symmetry of the modified bracket $[.,.]_E^{\nabla}$ yields the desired result.
\end{proof}
This proposition implies that there are no $E$-Levi-Civita connections for local almost-Leibniz algebroids which are not anti-commutable. On the other hand, it shows that assuming $E$-metric-$g$-compatibility in the definition of $E$-Levi-Civita connections corresponding to some generalized metric $G$ on a Courant algebroid $(E, \rho, [.,.]_E, g)$ is unnecessary as the $E$-torsion-freeness implies this $g$-compatibility already. Moreover, with the following proposition we prove that only possible way to be an $E$-Levi-Civita connection is being the unique $E$-Levi-Civita connection on the almost-Lie algebroid which arises due to the modified bracket.

\begin{Proposition} If $\nabla$ is an admissible linear $E$-connection, then 
\begin{align} 
2 g(\nabla_u v, w) = & \ 2 g(^g \tilde{\nabla}_u v, w) \nonumber \\
& - Q(\nabla, g)(u, v, w) - Q(\nabla, g)(v, u, w) + Q(\nabla, g)(w, u, v)  \nonumber\\
& - g(T(\nabla)(v, w), u) - g(T(\nabla)(u, w), v) + g(T(\nabla)(u, v), w).
\label{ec57}
\end{align}
where $^g \tilde{\nabla}$ is the unique $E$-Levi-Civita connection on the almost-Lie algebroid $(E, \rho, [\cdot,\cdot]_E^{\nabla})$.
\label{pc14}
\end{Proposition}

\begin{proof} This can be proven by copying the Schouten's trick for evaluating the symmetric and anti-symmetric parts of the $E$-connection. The anti-symmetric part is related to the $E$-torsion:
\begin{equation} \nabla_u v - \nabla_v u = T(\nabla)(u, v) + [u, v]_E^{\nabla} \quad \implies \quad g(\nabla_u v - \nabla_v u, w) = g(T(\nabla)(u, v), w) + g([u, v]_E^{\nabla}, w). \nonumber
\end{equation}
For the symmetric part we consider the combination
\begin{equation} Q(\nabla, g)(u, v, w) + Q(\nabla, g)(v, w, u) - Q(\nabla, g)(w, u, v), \nonumber
\end{equation}
which yields 
\begin{align}
g(\nabla_u v + \nabla_v u, w) &= \rho(u)(g(v, w)) + \rho(v)(g(u, w)) - \rho(w)(g(u, v) \nonumber\\
& \quad \ -g([v, w]_E^{\nabla}, u) - g([u, w]_E^{\nabla}, v) \nonumber\\
& \quad \ - Q(\nabla, g)(u, v, w) - Q(\nabla, g)(v, w, u) + Q(\nabla, g)(w, u, v), \nonumber\\
& \quad \ - g(T(\nabla)(v, w), u) - g(T(\nabla)(u, w), w). \nonumber
\end{align}
Note that in the last step the admissibility is crucial, because it is necessary to use the anti-symmetry of the modified bracket and the $E$-torsion operator. Combining the anti-symmetric and symmetric parts, and singling out the unique $E$-Levi-Civita connection $^g \tilde{\nabla}$ on the almost-Lie algebroid $(E, \rho, [.,.]_E^{\nabla})$ given by the Koszul formula yield the desired result.
\end{proof}

\noindent Note that similar points can be proven for projected $E$-torsion and projected $E$-Levi-Civita connections by making appropriate changes. On a local $E$-frame the equation (\ref{ec57}) becomes
\begin{align} 
\Gamma(\nabla)^a_{\ b c} = & \ \Gamma(^g \tilde{\nabla})^a_{\ b c} + \frac{1}{2} g^{a d} \Big[- Q(\nabla, g)_{b d c} + Q(\nabla, g)_{d c b} - Q(\nabla, g)_{c b d} \nonumber\\
& \qquad \qquad \qquad \qquad - g_{e c} \ T(\nabla)^e_{\ b d} + g_{e d} \ T(\nabla)^e_{\ b c} - g_{e b} \ T(\nabla)^e_{\ c d} \Big],
\label{ec58} 
\end{align}

One can extend the action of the brackets $[v,\cdot]_E, [v,\cdot]_E^{\nabla}$ and $[v,\cdot]_E^{\hat{\nabla}}$ to $E$-Leibniz derivatives $\mathcal{L}_v, \mathcal{L}_v^{\nabla}$ and $\hat{\mathcal{L}}_u^{\nabla}$ respectively in a completely analogous way to the Lie derivative. Action of these $E$-Leibniz derivatives agree on the smooth functions:
\begin{equation} \mathcal{L}_v f = \mathcal{L}_v^{\nabla} f = \hat{\mathcal{L}}_v^{\nabla} f = \rho(v)(f),
\label{ec59}
\end{equation}
for all $f \in C^{\infty}(M, \mathbb{R}), v \in \mathfrak{X}(E)$.

\begin{Proposition} For an $E$-Levi-Civita connection $\nabla$ corresponding to an $E$-metric $g$, the following holds
\begin{equation} \left( \mathcal{L}^{\nabla}_ v g \right) (u, w) = g(\nabla_u v, w) + g(u, \nabla_w v),
\label{ec60}
\end{equation}
for all $u, v, w \in \mathfrak{X}(E)$.
\label{pc15}
\end{Proposition}

\begin{proof} Let $\nabla$ be an $E$-Levi-Civita connection. By the proposition (\ref{pc13}) it is an admissible $E$-Koszul connection. Then, by definition it satisfies
\begin{align} 
2 g(\nabla_u v, w) = & \ \rho(u)(g(v, w)) + \rho(v)(g(u, w)) - \rho(w)(g(u, v)) \nonumber\\
												& - g([v, w]_E^{\nabla}, u) - g([u, w]_E^{\nabla}, v) + g([u, v]_E^{\nabla}, w), \nonumber
\end{align}
and 
\begin{align} 
2 g(u, \nabla_w v) = 2 g(\nabla_w v, u) = & \ \rho(w)(g(u, v)) + \rho(v)(g(w, u)) - \rho(u)(g(w, v)) \nonumber\\
												& - g([v, u]_E^{\nabla}, w) - g([w, u]_E^{\nabla}, v) + g([w, v]_E^{\nabla}, u). \nonumber
\end{align}
If one adds them together side by side, by using the symmetry of the $E$-metric and the anti-symmetry of the modified bracket for an admissible linear $E$-connection, one gets
\begin{align}
g(\nabla_u v, w) + g(u, \nabla_w v) &= \rho(v)(g(u, w)) - g([u, v]_E^{\nabla}, w) - g(u, [v, w]_E^{\nabla}) \nonumber\\
&= \left( \mathcal{L}^{\nabla}_v g \right) (u, w). \nonumber
\end{align}
\end{proof}

\noindent Note that this is valid for Levi-Civita connections in the usual metric-affine geometry.

\begin{Proposition} For an admissible linear $E$-connection $\nabla$, the following Cartan magic formulas hold for all $E$-$p$-forms:
\begin{align} 
&\mathcal{L}_v^{\nabla} = d(\nabla) \iota_v + \iota_v d(\nabla), \nonumber\\
&\hat{\mathcal{L}}_v^{\nabla} = \hat{d}(\nabla) \iota_v + \iota_v \hat{d}(\nabla).
\label{ec61}
\end{align}
\label{pc16}
\end{Proposition}

\begin{proof} We will prove only the first one because the proofs are exactly the same. We will show it only for an $E$-1-form $\Omega$ but it is valid for any $E$-$p$-form, which can be proven by induction:
\begin{align} 
(d(\nabla) \iota_v \Omega)(u)+ (\iota_v d(\nabla) \Omega)(u) &= d(\Omega(v))(u) + d \Omega (v, u) \nonumber\\
&= \rho(u)(\Omega(v)) + \left\{ \rho(v)(\Omega(u)) - \rho(u)(\Omega(v)) - \Omega([v, u]_E^{\nabla}) \right\} \nonumber\\
&= \rho(v)(\Omega(u)) - \Omega([v, u]_E^{\nabla}) \nonumber\\
&= \left( \mathcal{L}^{\nabla}_v \Omega \right) (u). \nonumber
\end{align}
\end{proof}

\begin{Corollary} For an admissible linear $E$-connection $\nabla$,
\begin{align} 
&\mathcal{L}_{f v}^{\nabla} \Omega = f \mathcal{L}_v^{\nabla} \Omega + d(\nabla) f \wedge \iota_v \Omega, \nonumber\\
&\mathcal{L}_{f v}^{\hat{\nabla}} \Omega = f \mathcal{L}_v^{\hat{\nabla}} \Omega + \hat{d}(\nabla) f \wedge \iota_v \Omega,
\end{align}
for all $v \in \mathfrak{X}(E), \Omega \in \Omega^p(E), f \in C^{\infty}(M, \mathbb{R})$. 
\label{cc4}
\end{Corollary}

\begin{proof} This is a direct consequence of the Cartan magic formulas (\ref{ec61}).
\end{proof}



\begin{Proposition} The following formulas hold
\begin{align} 
&\mathcal{L}_u \iota_v - \iota_v \mathcal{L}_u = \iota_{[u, v]_E}, \nonumber\\
&\mathcal{L}_v (\Omega_1 \wedge \Omega_2) = \mathcal{L}_v \Omega_1 \wedge \Omega_2 + \Omega_1 \wedge \mathcal{L}_v \Omega_2, 
\label{ec62}
\end{align}
for all $u, v \in \mathfrak{X}(E), \Omega_i \in \Omega^{p_i}(E)$. 
\label{pc17}
\end{Proposition}

\begin{proof} Again, we will give the proof for an $E$-1-form $\Omega$ and the most general case can be proven by using induction.
\begin{align}
\left( \mathcal{L}_u \iota_v \right) \Omega - \left( \iota_v \mathcal{L}_u \right) &= \mathcal{L}_u (\Omega(v)) - \iota_v \left\{ \rho(u) \Omega(.) - \Omega([u, .]_E^{\nabla}) \right\} \nonumber\\
&= \rho(u)(\Omega(v)) - \rho(u)(\Omega(v)) + \Omega([u, v]_E^{\nabla}) \nonumber\\
&= \Omega([u, v]_E^{\nabla}) \nonumber\\
&= \iota_{[u, v]_E^{\nabla}} \Omega. \nonumber
\end{align}
We will skip the proof of the second identity.
\end{proof}

\noindent This proposition is also valid for modified and projected modified brackets with their corresponding $E$-Leibniz derivatives. Moreover, as they are anti-symmetric for an admissible linear $E$-connection, one gets
\begin{align} 
&\mathcal{L}_v^{\nabla} \iota_v = \iota_v \mathcal{L}_v^{\nabla}, \nonumber\\
&\mathcal{L}_v^{\hat{\nabla}} \iota_v = \iota_v \mathcal{L}_v^{\hat{\nabla}}.
\label{ec63}
\end{align}

Note that all of these formulas (\ref{ec61}, \ref{ec62} \ref{ec63}) are valid for the usual metric-affine geometry. There are other important formulas which hold for the manifold case:
\begin{align} 
&\mathcal{L}_U \mathcal{L}_V - \mathcal{L}_V \mathcal{L}_U = \mathcal{L}_{[U, V]}, \nonumber\\
&\mathcal{L}_V d = d \mathcal{L}_V,
\label{ec64}
\end{align}
for all usual vector fields $U$ and $V$. These do not hold for the brackets $[.,.]_E, [.,.]_E^{\nabla}$ or $[.,.]_E^{\hat{\nabla}}$ in general as they do not satisfy the Jacobi identity. When their associator is 0, these formulas also hold for them.

The set of all $E$-forms $\Omega(E) := \oplus_{p = 0}^{rank(E)} \Omega^p(E)$ is a graded commutative algebra. Moreover, the set of all graded derivations on $\Omega(E)$, denoted by $Der(\Omega(E)) := \oplus_{k} Der^k(\Omega(E))$ where $Der^k(\Omega(E))$ is the set of graded derivations of degree-$k$, becomes a graded Lie algebra with the graded Lie bracket:

\begin{equation} [D_1, D_2]_{der} := D_1 D_2 - (-1)^{k_1 k_2} D_2 D_1,
\label{ec65}
\end{equation}
where $D_i \in Der^{k_i}(\Omega(E))$. So far, it has been mentioned that 
\begin{align} 
\iota_v &\in Der^{-1}(\Omega(E)), \nonumber\\
\mathcal{L}_v, \mathcal{L}_v^{\nabla}, \hat{\mathcal{L}}_v^{\nabla} &\in Der^0(\Omega(E)), \nonumber\\
d(\nabla), \hat{d}(\nabla) &\in Der^1(\Omega(E)),
\label{ec66}
\end{align}
for all $v \in \mathfrak{X}(E)$ and admissible linear $E$-connection $\nabla$. Some of the previously proven properties can be written by using the graded Lie bracket on the set of graded derivations:

\begin{align} 
\hat{d}(\nabla)^2 = 0 \qquad &\iff \qquad \frac{1}{2} [\hat{d}(\nabla), \hat{d}(\nabla)]_{der} = 0, \nonumber\\
\hat{d}(\nabla) \iota_v + \iota_v \hat{d}(\nabla) = \hat{\mathcal{L}}_v^{\nabla} \qquad &\iff \qquad [\hat{d}(\nabla), \iota_v]_{der} = \hat{\mathcal{L}}_v^{\nabla}, \nonumber\\
\mathcal{L}_u \iota_v - \iota_v \mathcal{L}_u = \iota_{[u, v]_E} \qquad &\iff \qquad [\mathcal{L}_u, \iota_v]_{der} = \iota_{[u, v]_E}, \nonumber\\
\mathcal{L}_v^{\nabla} \iota_v - \iota_v \mathcal{L}_v^{\nabla} = 0 \qquad &\iff \qquad [\mathcal{L}_v^{\nabla}, \iota_v]_{der} = 0. \nonumber\\
\label{ec67}
\end{align}
With these final observations, we believe that a comprehensive analogy between the usual metric-affine geometries and metric-connection geometries on anti-commutable pre-Leibniz algebroids for an admissible linear $E$-connection is established. We proved most of the usual propositions that hold in the manifold case, where even the proofs are completely parallel.


\section{Concluding Remarks}

In this paper, we introduced and then investigated geometrical structures on anti-commutable pre-Leibniz algebroids. Many algebroids studied in the literature are special cases for these anti-commutable pre-Leibniz algebroids. Thus, the presented geometrical framework on them automatically applies to a variety of different cases. The most crucial point of this work is the definition of admissible linear $E$-connections according to which the bracket satisfies a property like anti-commutativity (\ref{ec8}). We claim that only such admissible linear $E$-connections should be considered while constructing a geometry on an algebroid. We prove that the admissibility condition is equivalent to a metric-compatibility condition in some generalized sense for many existing algebroids in the literature. This is the reason why Courant algebroid connections are necessarily $E$-metric-compatible in order to carry useful properties. In terms of the admissible linear $E$-connections, we offer a new point of view that the modification of $E$-torsion and $E$-curvature operators should be considered as a modification of the bracket instead. Hence one may use the beneficial features of almost- and pre-Lie algebroids, and prove many of the desirable properties and relations that hold in usual metric-affine geometries on a smooth manifold. These include the first and second Bianchi identities, the Ricci identity, Cartan structure equations, Cartan magic formula and the decomposition of linear $E$-connections in terms of its $E$-torsion and $E$-non-metricity. 

There are several important structures in the setting of manifolds that can be studied at the algebroid level. For instance, Weyl invariant theories \cite{28} on an anti-commutable pre-Leibniz algebroid seems a promising case. We actually dealt with conformal, projective and Weyl structures on pre-Leibniz algebroids, which led us to the current work of this paper because many crucial properties of these structures heavily depend on the anti-commutativity of the Lie bracket. Hence, working on an anti-commutable pre-Leibniz algebroid sustains a natural framework for these structures. We plan to investigate different compatibility conditions for these structures as in the work of Matveev and Scholz in the usual geometrical setting \cite{29}. A fruitful way to proceed might be to combine these ideas and Brans-Dicke theories \cite{30} since the dilaton field is an important ingredient of string theories so that it is naturally geometrized in an algebroid setting. In future, we wish to work comprehensively on this topic so that the dependence on the light-cone structure can be relaxed due to the existence of a dilaton. Moreover, as the $E$-curvature map is anti-symmetric, one can consider it as an $End(E)$-valued $E$-2-form. This $E$-2-form would make it possible to work on a pre-quantization scheme on algebroids by introducing a line bundle with an $E$-connection whose $E$-curvature 2-form is symplectic. Additionally, in the near future we also plan to work on the algebroid version of statistical structures and information geometry \cite{31}. More importantly, we plan to investigate the admissibility condition on different algebroids, including the $AV$-Courant algebroids \cite{32}, omni-Lie algebroids \cite{10}, $E$-Courant algebroids \cite{33} and $G$-algebroids \cite{34}. Especially, $E$-Courant algebroids look promising as they are constructed by using the jet and covariant differential operator bundles for some vector bundle, where the latter can provide a framework for generalizations of the metric-compatibility conditions which our admissibility condition fits naturally.  

\section{Acknowledgments}

The authors are thankful to Cem Yetişmişoğlu for long and fruitful discussions on many details of this work, especially on the proof of the proposition (\ref{pc5}).

\newpage


\end{document}